\documentclass[11pt,reqno,twoside]{article}
\usepackage{mysty}
\newlist{Tenumerate}{enumerate}{1}
\setlist[Tenumerate]{label=\rm($T_\arabic*$), ref=\rm($T_\arabic*$)}
\newlist{renumerate}{enumerate}{1}
\setlist[renumerate]{label=(\roman*), ref=(\roman*)}

\newcommand{\textoperatorname}[1]{%
  \operatorname{\textnormal{#1}}%
}	
\DeclareMathOperator{\Loj}{\textoperatorname{Ł}}

\usepackage{wrapfig}
\usepackage{mdwlist}

\usepackage{fullpage}
\usepackage{patchcmd,pifont}


\usepackage{enumitem}
\SetLabelAlign{parright}{\parbox[t]{\labelwidth}{\raggedleft{#1}}}
\setlist[description]{style=multiline,topsep=4pt,align=parright}%,font=\normalfont

\makeatletter
\let\reftagform@=\tagform@
\def\tagform@#1{\maketag@@@{(\ignorespaces\textcolor{black}{#1}\unskip\@@italiccorr)}}
\newcommand{\iref}[1]{\textup{\reftagform@{\tcr{\ref{#1}}}}}
\makeatother

\begin{document}
%%%%%%%%%%%%%%%%%%%%%%%%%%%%%%%%%%%
\title{Tikhonov Regularization for Stochastic Non-Smooth Convex Optimization in Hilbert Spaces}
\author{Rodrigo Maulen-Soto\thanks{Normandie Universit\'e, ENSICAEN, UNICAEN, CNRS, GREYC, France. E-mail: rodrigo.maulen@ensicaen.fr} \and
Jalal Fadili\thanks{Normandie Universit\'e, ENSICAEN, UNICAEN, CNRS, GREYC, France. E-mail: Jalal.Fadili@ensicaen.fr} \and Hedy Attouch\thanks{IMAG, CNRS, Universit\'e Montpellier, France. E-mail: hedy.attouch@umontpellier.fr}
}
\date{}
\maketitle
%\begin{flushleft}\end{flushleft}
%%%%%%%%%%%%%%%%%%%%%%%%%%%%%%%%%%%

%%%%%%%%%%%%%%%%%%%%%%%%%%%%%%%%%%%
\begin{abstract} 
 To solve convex optimization problems with a noisy gradient input, we analyze the global behavior of
subgradient-like flows under stochastic errors. The objective function is composite, being equal to the sum of two
convex functions, one being differentiable and the other potentially non-smooth. We then use stochastic differential
inclusions where the drift term is minus the subgradient of the objective function,
and the diffusion term is either bounded or square-integrable. In this context, under Lipschitz's continuity of the differentiable
term and a growth condition of the non-smooth term, our first main result shows almost sure weak convergence
of the trajectory process towards a minimizer of the objective function. Then, using Tikhonov regularization with a properly tuned vanishing parameter, we can obtain almost sure strong convergence
of the trajectory towards the minimum norm solution. We find an explicit tuning of this parameter when our objective
function satisfies a local error-bound inequality. We also provide a comprehensive complexity analysis by establishing
several new pointwise and ergodic convergence rates in expectation for the convex, strongly convex, and {\L}ojasiewicz case.
\end{abstract}



\section{Introduction}\label{sec:intro}
\subsection{Problem statement} We aim to solve convex minimization problems by means of stochastic differential inclusions (SDI), showing the existence, uniqueness, and properties of the solution. Then, we work with Tikhonov regularization, specifically when the drift term is the sum of the (sub-)gradient of the objective function and of a Tikhonov regularization term with a vanishing coefficient. This makes it possible to take into account a noisy (imprecise) gradient input and obtain convergence a.s. to the minimal norm solution.

Let %us make our assumptions precise and 
us consider the minimization problem
\begin{equation}\label{P}\tag{P}
    \min_{x\in \H} F(x)\eqdef f(x)+g(x),
\end{equation}
where $\H$ is a separable real Hilbert space, and the objective $F$ satisfies the following standing assumptions:
\begin{align}\label{H0}
\begin{cases}
\text{$f:\H\rightarrow\R$ is continuously differentiable and convex with $L$-Lipschitz continuous gradient}; \\
\text{$g:\H\rightarrow\R$ is proper, lsc and convex};\\
\calS_F\eqdef \argmin (F)\neq\emptyset. \tag{$\mathrm{H}_0$} 
\end{cases}
\end{align} 

To solve \eqref{P}, a fundamental dynamic to consider is the subgradient flow, which is the following differential inclusion (DI) starting in $t_0\geq 0$ with initial condition $x_0\in\H$:
\begin{equation}
\begin{cases}\label{DI}\tag{DI}
\begin{aligned}
\dot{x}(t)&\in-\partial F(x(t)),\quad t>t_0;\\
x(t_0)&=x_0.
\end{aligned}
\end{cases}
\end{equation}

 It is well known since the founding articles of Br\'ezis, Baillon, Bruck in the 1970s that,   when the initial data $x_0$ is in the domain of $F$, (more generally when it is in its closure), there exists a unique strong global solution of \eqref{DI}.
Moreover, if the solution set $\argmin (F)$ of \eqref{P} is nonempty then each solution trajectory of \eqref{DI} converges weakly, and its limit belongs to $\argmin (F)$.

\smallskip

In many cases, the gradient input is subject to noise, for example, if the gradient cannot be evaluated directly, or due to some other exogenous factor. In such scenario, one can model the associated errors using a stochastic integral with respect to the measure defined by a continuous It\^o martingale. This entails the following stochastic differential inclusion (SDI) as a stochastic counterpart of \eqref{DI}
\begin{equation}\label{SDI0}\tag{$\mathrm{SDI}$}
\begin{cases}
\begin{aligned}
dX(t)&\in -\partial F(X(t)) +\sigma(t,X(t))dW(t), \quad t\geq t_0;\\
X(t_0)&=X_0,
\end{aligned}
\end{cases}
\end{equation}
where the diffusion (volatility) term $\sigma:[t_0,+\infty[\times\H\rightarrow \calL_2(\K;\H)$ (see notation in Section \ref{sec:notation}) is a measurable function, $\K$ a separable real Hilbert space, and $W$ is a $\K$-valued cylindrical Brownian motion (see Section~\ref{onstochastic} for a precise definition), and the initial data $X_0$ is a properly measurable $\H$-valued random variable. This dynamic can be viewed as a stochastic dissipative system that aims to minimize $F$ if the diffusion term vanishes sufficiently fast. Also, it is the natural extension to the non-smooth setting of the work done in \cite{mio}. \smallskip

%To show that the inclusion \eqref{SDI0} is well-posed, we will build upon the work of \cite{petterson}, which showed the existence of a solution of \eqref{SDI0} in order to show other properties such as the uniqueness; and conditions under which one has almost sure weak convergence to the set of minimizers.\smallskip

%%%%TIKHONOV%%%%%%%%%%

An important aspect of our work concerns the Tikhonov regularization of \eqref{DI} and \eqref{SDI0}. 
Given $t_0>0$, and a regularization parameter $\varepsilon : [t_0, +\infty[ \to \R_+$, which is a measurable function that vanishes asymptotically in a controlled way,  the Tikhonov regularization of \eqref{DI} is written:
\begin{equation}
\begin{cases}\label{DITA}\tag{DI-TA}
\begin{aligned}
\dot{x}(t)&\in-\partial F(x(t))-\varepsilon(t) x(t),\quad t>t_0;\\
x(t_0)&=x_0.
\end{aligned}
\end{cases}
\end{equation}
The stochastic counterpart of \eqref{DITA} (which is the Tikhonov regularization of \eqref{SDI0}), is the following stochastic differential inclusion with initial data $X_0\in\Lp^{\nu}(\Omega;\H)$ (for some $\nu\geq 2$):
\begin{equation}\label{CSGD}\tag{$\mathrm{SDI-TA}$}
\begin{cases}
\begin{aligned}
dX(t)&\in-\partial F(X(t)) -\varepsilon(t) X(t) +\sigma(t,X(t))dW(t), \quad t> t_0;\\
X(t_0)&=X_0.
\end{aligned}
\end{cases}
\end{equation}

It is well-known that in the deterministic case of \eqref{DITA}, the Tikhonov regularization ensures that the trajectory generated by the system converges strongly to a particular minimizer of $F$: the one of minimum norm; see \cite{AC96,CPS} and references therein. The fact that the Tikhonov regularization parameter $\varepsilon (t)$ tends to zero not too fast as $t \to +\infty$ induces a hierarchical minimization property: the limit of any trajectory no longer depends on the initial data, it is precisely the minimum norm solution. 

It is our aim in this paper to extend these results to the stochastic case \eqref{CSGD} based on the recent work of Maulen-Soto, Fadili, and Attouch \cite{mio}. More precisely, our objective is to study the dynamics \eqref{SDI0} and \eqref{CSGD} and their long-time behavior in order to solve \eqref{P}. If the diffusion term vanishes with time, one would expect to solve \eqref{P} with our dynamics and obtain for \eqref{CSGD} the hierarchical minimization property described above. %We refer the reader to \cite{mio} for some observations on why this happens.

Motivated by this, our paper will primarily focus on the case where $\sigma(\cdot,x)$ vanishes sufficiently fast as $t \to +\infty$ uniformly in $x$. Additionally, we will provide some guarantees for uniformly bounded $\sigma$. Therefore, throughout the paper, we assume that $\sigma$ satisfies:
\begin{equation*}
\tag{$\mathrm{H}$}\label{H}
\begin{cases}
\sup_{t \geq t_0,x \in \H} \Vert \sigma(t,x)\Vert_{\HS}<+\infty, \\
\Vert \sigma(t,x')-\sigma(t,x)\Vert_{\HS} \leq L_0\norm{x'-x},
\end{cases}
\end{equation*}   
for some $L_0>0$ and for all  $t\geq t_0, x, x'\in \H$ (where the $\HS$-norm is defined in Section~\ref{sec:notation}). The Lipschitz continuity assumption is mild and required to ensure the well-posedness of \eqref{SDI0} and \eqref{CSGD}. 

\subsection{Contributions}

%\todo{The introduction is far too short. You have to elaborate much more on the contributions, to put your work in perspective w.r.t. literature, and to discuss existing work on several aspects: SDEs (shortened version of the discussion in the SDE paper, on SDIs, etc). Subection~\ref{problem_statement} and Subection~\ref{sec:sdes} must be removed from the introduction and probably put in the preliminaries (I am not sure btw that Subection~\ref{sec:sdes} should be kept as it is a mere repetition of the SDE paper.}

This work goes well beyond that of \cite{mio} in three directions: we consider the non-smooth case, in infinite dimensional Hilbert spaces, and with Tikhonov regularization. The latter makes it possible to pass from weak convergence to strong convergence, and to a particular solution, that of minimal norm.

\smallskip

We first study the properties of the process $X(t)$ and $F(X(t))$ for the stochastic differential inclusion \eqref{SDI0} on separable real Hilbert spaces from an optimization perspective, under the assumptions \eqref{H0}, \eqref{H} and \eqref{Hl} (introduced in Section \ref{sec:sdi}). When the diffusion term is uniformly bounded, we show convergence of $\EE[F(X(t))-\min F]$ to a noise-dominated region both for the convex and strongly convex case. When the diffusion term is square-integrable, we show in Theorem \ref{converge} that $X(t)$ weakly converges almost surely to a solution of \eqref{P}, which is a new result to the best of our knowledge. Moreover, in Theorem \ref{importante0}, we provide new ergodic and pointwise convergence rates of the objective in expectation, again, for both the convex and strongly convex case. \smallskip

Next, we consider \eqref{CSGD}, obtained by adding a Tikhonov regularization term to \eqref{SDI0}. We show in Theorem~\ref{converge20} that under certain conditions on the regularization term, $X(t)$ strongly converges almost surely to the minimum norm solution. Then, we show in Theorem \ref{practical} some practical situations where one can obtain an explicit form of the Tikhonov regularizer. Moreover, in Theorem \ref{importante1}, we show new convergence rates of the objective and the trajectory in expectation for the smooth case. \smallskip

Table~\ref{table:summary} summarizes the convergence rates obtained for $\EE[F(X(t))-\min F]$. We use the following notation, $F=f+g$, $\sigma_*>0$ and $\sigma_{\infty}(\cdot)$ is defined as
\begin{equation}\label{eq:defsigstar}
\sigma_{\infty}(t)\eqdef \sup_{x\in\H}\norm{\sigma(t,x)}_{\HS}, \qwhereq \norm{\sigma(t,x)}_{\HS}^2\leq \sigma_*^2, \quad \forall t\geq t_0,\forall x\in\H . 
\end{equation}
%
%and $\sigma_{\infty}(\cdot)$ is a non-increasing function. 

\begin{table}[H]
\begin{center}
\begin{tabular}{|c|c|c|cc|}
\hline
\textbf{Property of $F$} & \textbf{DI} & \textbf{SDI $(\sup_{t\geq t_0}\sigma_{\infty}(t)\leq\sigma_*)$} & \multicolumn{2}{c|}{\textbf{SDI  $(\sigma_{\infty}\in \Lp^2([t_0,+\infty[))$}}    \\ \hline
Convex & $t^{-1}$ & $t^{-1}+\sigma_*^2$ & \multicolumn{2}{c|}{$t^{-1}$}    \\ \hline
$\mu$-Strongly Convex & $e^{-2\mu t}$ & $e^{-\mu t}+\sigma_*^2$ & \multicolumn{2}{c|}{$\max\{e^{-\mu t},\sigma_{\infty}^2(t)\}$}  \\ \hline
\end{tabular}
\caption{Summary of convergence rates obtained for $\EE[F(X(t))-\min F]$.}
\label{table:summary}
\end{center}
\end{table}
We also denote $\EB^p(\calS)$ the local Error Bound Inequality defined in \eqref{eq:errbnd}. In Table~\ref{table:summary2}, we summarize the results obtained in the smooth case for the dynamics with Tikhonov regularization, \ie, when $g\equiv 0$. 
\begin{table}[H]
    \begin{center}
    \begin{tabular}{|c|c|c|}
    \hline
    \textbf{Property of $f$}
         & \textbf{DI-TA ($\varepsilon(t)=t^{-r}, r\in ]0,      1[$)} & \textbf{SDI-TA $\left(\varepsilon(t)=t^{-r}, r\in ]\frac{2p}{2p+1},1[\right)$} \\ \hline
        $\text{Convex} \cap \EB^p(\calS)$ & $t^{-r}$ & $t^{-r}$ whenever $\sigma_{\infty}^2(t)=\mathcal{O}(t^{-2r})$. \\ \hline
    \end{tabular}
    \caption{Summary of convergence rates obtained for $\EE[f(X(t))-\min f]$ for the dynamics with Tikhonov regularization when $\varepsilon(t)=t^{-r}$}
    \label{table:summary2}
    \end{center}
\end{table}

\subsection{Relation to prior work}

The subgradient flow dynamic \eqref{DI}, which is valid on a general real Hilbert space, is a dissipative dynamical system, whose study dates back to Cauchy \cite{Cauchy}. It plays a fundamental role in optimization: it transforms the problem of minimizing $F$ into the study of the asymptotic behavior of the trajectories of \eqref{DI}. Its Euler forward discretization (with stepsize $\gamma_k>0$) is the subgradient method 
\begin{equation}\label{SubG}\tag{Sub-G}
    x_{k+1}\in x_k-\gamma_k \partial  F(x_k).
\end{equation} 
Or equivalently, 
\begin{equation}
    x_{k+1}= x_k-\gamma_k g_k,
\end{equation} 
where $g_k\in\partial  F(x_k)$ for every $k\in \N$. %Since the subgradient method is not a descent method, we keep track of its descent by defining $$f_{k}^{\text{best}}\eqdef \min\{f_{k-1}^{\text{best}},f(x_k)\}.$$ 

Let us focus on the finite-dimensional case ($\H=\R^d$). %Under \eqref{H0}, and for $\gamma_k=h$ (with $h$ small enough) or $\gamma_k=\frac{h}{\Vert g_k\Vert}$, then we can guarantee (see 
In \cite{shor, akgul} they give conditions on the function and the stepsize to converge to within some range of the optimal value and to the optimal value. Despite \eqref{SubG} being a classical algorithm to solve the non-smooth convex minimization problem, it is not recommended for general use, as discussed in \cite{sub1,sub2}. Moreover, with the need to handle large-scale problems (such as in various areas of data science and machine learning), it has become necessary to find ways to get around the high computational cost per iteration that these problems entail. The Robbins-Monro stochastic approximation algorithm \cite{rob} is at the heart of Stochastic Gradient Descent methods, which, roughly speaking, consists in cheaply and randomly approximating the gradient at the price of obtaining a random noise in the solutions. In \cite{gwinner} they propose the natural generalization to the non-smooth setting, the stochastic subgradient method \eqref{ssug1} that updates the iterates according to \begin{equation}\tag{S-Sub-G}\label{ssug1}
    x_{k+1}\in x_k-\gamma_k (\partial  F(x_k)+\xi_k),
\end{equation} 
where $\xi_k$ denotes the (random) noise term on the subgradient at the $k$-th iteration, and $\EE[\xi_k]=0$.\smallskip

The SDI continuous-time approach is motivated by its relations to \eqref{ssug1}, where the latter can be viewed as an Euler forward time discretization, and the noise $\xi_k\sim\mathcal{N}(0,\sigma_k I_d)$ (hence not necessarily bounded). The advantage of the continuous-time perspective is that it offers a deep insight and unveils the key properties of the dynamic, without being tied to a specific discretization.\smallskip

%The main influence is \cite{mio,petterson, ACR}\smallskip

%There are many definitions of solution for a stochastic differential inclusion, this work uses the definition of \cite{petterson, govindan}. Moreover, 
We extend the work of \cite{mio} to the case where the objective is ``smooth+non-smooth'', being able to show the almost sure weak convergence of the trajectory to the set of minimizers and new convergence rates for the objective in the convex and strongly convex case.

Besides, based on the work of \cite{ACR}, we add a Tikhonov term that let us obtain the almost sure strong convergence of the trajectory to the minimal norm solution. Moreover, we extend the convergence rates shown in \cite[Theorem~5]{ACR} to the stochastic case. In our way, we even prove new and useful results for the deterministic setting (e.g., Proposition~\ref{ratetikhonov} and Corollary~\ref{cor:model-a}).

While the use of Lyapunov analysis and vanishing Tikhonov regularization are known techniques in the deterministic case~\cite{ACR}, their adaptation to the stochastic setting requires significant technical work and novel arguments. One has not only to handle carefully stochasticity through proper It\^o's calculus, but also non-smoothness of the objective function.

\subsection{Organization of the paper}

Section~\ref{sec:notation} introduces notations and reviews some necessary material from convex and stochastic analysis. Section~\ref{sec:sdi} states our main convergence results of \eqref{SDI0} in the case of a convex objective function under \eqref{H0} and with an extra assumption on the non-smooth term. We first show the almost sure weak convergence of the process towards the set of minimizers when the diffusion term is square-integrable, then we establish convergence rates for the values. Section~\ref{sec:tikhonov} introduces an extra vanishing term called Tikhonov regularizer that let us obtain the almost sure strong convergence of \eqref{CSGD} to the minimal norm solution. Then we give some practical situations where we can obtain an explicit tuning of the Tikhonov regularizer. Finally in this section, we present convergence rates for the values and for the trajectory in the smooth case. Technical lemmas and theorems that are needed throughout the paper will be collected in the appendix \ref{aux}.

%%%%%%%%%%NOTATION%%%%%%%%%%%%

\section{Notation and Preliminaries}\label{sec:notation}

We will use the following shorthand notations:  Given $n\in\N$,  $[n]\eqdef \{1,\ldots,n\}$. {Consider $\H,\K$ real separable Hilbert spaces endowed with the inner product $\langle\cdot,\cdot\rangle_{\H}$ and $\langle\cdot,\cdot\rangle_{\K}$, respectively, and norm $\Vert \cdot\Vert_{\H}=\sqrt{\langle \cdot,\cdot \rangle_{\H}}$ and $\Vert \cdot\Vert_{\K}=\sqrt{\langle \cdot,\cdot \rangle_{\K}}$, respectively (we omit the subscripts $\H$ and $\K$ for the sake of clarity). $I_{\H}$ is the identity operator from $\H$ to $\H$. $\calL(\K;\H)$ is the space of bounded linear operators from $\K$ to $\H$, $\calL_1(\K)$ is the space of trace-class operators, and $\calL_2(\K;\H)$ is the space of bounded linear Hilbert-Schmidt operators from $\K$ to $\H$}. For $M\in\calL_1(\K)$, is trace is defined by
\[
\tr(M)\eqdef \sum_{i\in I} \langle Me_i,e_i\rangle<+\infty,
\]
where $I\subseteq \N$ and $(e_i)_{i\in I}$ is an orthonormal basis of $\K$. Besides, for $M\in\calL(\K;\H)$, $M^{\star}\in\calL(\H;\K)$ is the adjoint operator of $M$, and for $M\in\calL_2(\K;\H)$, 
\[
\norm{M}_{\mathrm{HS}}\eqdef \sqrt{\tr(MM^{\star})}<+\infty
\] 
is its Hilbert-Schmidt norm (in the finite-dimensional case is equivalent to the Frobenius norm). We denote by $\wlim$ (resp. $\slim$) the limit for the weak (resp. strong) topology of $\H$. The notation $A: \H\rightrightarrows \H$ means that $A$ is a set-valued operator from $\H$ to $\H$. Consider $f:\H\rightarrow\R$, the sublevel of $f$ at height $r\in\R$ is denoted $[f\leq r]\eqdef \{x\in{\H}: f(x)\leq r\}$. For $1 \leq p \leq +\infty$, $\Lp^p([a,b])$ is the space of measurable functions $g:\R\rightarrow\R$ such that $\int_a^b|g(t)|^p dt<+\infty$, with the usual adaptation when $p = +\infty$. %Functions obeying  $\int_{t_0}^{+\infty} |g(t)|^p dt<+\infty$ belong to $\Lp^p([t_0,+\infty[)$.
On the probability space $(\Omega,\calF,\PP)$, $\Lp^p(\Omega;\H)$ denotes the (Bochner) space of $\H$-valued random variables whose $p$-th moment (with respect to the measure $\PP$) is finite. Other notations will be explained when they first appear.

Let us recall some important definitions and results from convex analysis; for a comprehensive coverage, we refer the reader to \cite{rocka}.

\smallskip

We denote by $\Gamma_0(\H)$ the class of proper lsc and convex functions on $\H$ taking values in $\Rinf$.
%which are finitely real valued \footnote{In Section \ref{nonsmooth_structured} we will consider the more general situation of convex functions which are lower semicontinuous with real extended values}. 
For $\mu > 0$, $\Gamma_{\mu}(\H) \subset \Gamma_0(\H)$ is the class of $\mu$-strongly convex functions, roughly speaking, this means that there exists a quadratic lower bound on the growth of these functions. We denote by $C^{s}(\H)$ the class of $s$-times continuously differentiable functions on $\H$. 
%$C^{1,1}(\H)$ is the space of functions from $\H$ to $\R$ which are continuously differentiable and whose gradient is Lipschitz continuous. 
For $L \geq 0$, $C_L^{1,1}(\H) \subset C^{1}(\H)$ is the set of functions on $\H$ whose gradient is $L$-Lipschitz continuous, and $C_L^2(\H)$ is the subset of $C_L^{1,1}(\H)$ whose functions are twice differentiable.

The \textit{subdifferential} of a function $f\in\Gamma_0(\H)$ is the set-valued operator $\partial f:\H \rightrightarrows \H$ such that, for every $x$ in $\H$,
\[
\partial f(x)=\{u\in\H:f(y)\geq f(x) + \dotp{u}{y-x} \qforallq y\in\H\}.
\]
% We denote $\dom(\partial f)\eqdef \{x\in\H:\partial f(x)\neq\emptyset\}$.
When $f$ is continuous, $\partial f(x)$ is a non-empty convex and compact set for every $x\in \H$. If $f$ is differentiable, then $\partial f(x)=\{\nabla f(x)\}$. For every $x\in \H$ such that $\partial f(x) \neq \emptyset$, the minimum norm selection of $\partial f(x)$ is the unique element $\{\partial^0 f(x)\} \eqdef \argmin_{u\in \partial f(x)}\norm{u}$. 

\smallskip

The projection of a point $x\in \H$ onto a closed convex set $C\subseteq \H$ is denoted by $\proj_C(x)$.

\subsection{Deterministic results on the subgradient flow with Tikhonov regularization}\label{problem_statement}

Let us first recall some basic facts about the deterministic case. To solve \eqref{P}, a fundamental dynamic to consider is the subgradient flow of $F$, \ie the following differential inclusion:
\begin{equation}\label{GF}\tag{DI}
    \dot{x}(t)\in - \partial F(x(t)).
\end{equation}

It is well known since the founding papers of Br\'ezis, Baillon, and Bruck in the 1970s that, if the solution set $\argmin (F)$ of \eqref{P} is non-empty and $F$ is convex, lower semicontinuous (lsc) and proper, then each solution trajectory of \eqref{GF} converges weakly, and its weak limit belongs to $\argmin (F)$.

In general, the limit solution depends on the initial data and is a priori difficult to specify when one has a set of solutions not reduced to only one element. To remedy this difficulty we consider 
the differential inclusion with vanishing Tikhonov regularization, $\varepsilon (t) \to 0$ (denoted (DI-TA))   which gives
\begin{equation}\label{CSGDT1}\tag{$\mathrm{DI-TA}$}
\dot{x}(t) + \partial F(x(t)) + \varepsilon (t) x(t) \ni 0.
\end{equation}

To analyze the convergence properties of this dynamic, let us recall basic facts concerning the Tikhonov approximation (1963). It consists in approximating the convex minimization problem (possibly ill-posed)
\begin{center}
$
(\mathcal P) \quad  \min \left\lbrace  F(x) : \ x \in \mathcal H \right\rbrace,
$
\end{center}
\vspace{-1mm}
by the strongly convex minimization problem ($\varepsilon >0$)
\begin{center}
$
(\mathcal P)_{\varepsilon} \quad  \min \left\lbrace  F (x) + \dfrac{\varepsilon}{2}\|x\|^{2} : \ x \in \mathcal H \right\rbrace
$
\end{center}
\vspace{-1mm}
whose unique solution is denoted by $x_{\varepsilon}$.
The following result was first obtained by Browder in 1966 \cite{Browder1,Browder2}.

\begin{theorem}(Hierarchical minimization). \label{thm:hierarmin}
Suppose  that $\calS_F=\argmin (F) \neq \emptyset $. Let $x^{\star}=\proj_{\calS_F} (0)$. Then,
\begin{renumerate}
\item $\|x_{\varepsilon } \| \leq \|x^{\star}\|  \mbox{  for all } \varepsilon  >0$.
\item $\lim_{\varepsilon \rightarrow 0}\|x_{\varepsilon }-x^{\star}\|=0$.
\end{renumerate}
\end{theorem}

\medskip

The system \eqref{CSGDT1} is a special case of the general dynamic model
\begin{equation}\label{HBF-multiscale}
\dot{x}(t) + \nabla F (x(t)) + \varepsilon (t) \nabla \Psi (x(t))\ni0
\end{equation}
which involves two  functions $F$ and $\Psi$ intervening with different time scale. When $\varepsilon (\cdot)$ tends to zero moderately slowly, it was shown in \cite{czarnecki} that the trajectories of (\ref{HBF-multiscale})  converge asymptotically to equilibria that are solutions of the following hierarchical problem: they minimize the function $\Psi$ on the set of minimizers of $F$.
The continuous and discrete-time versions of these systems have a natural connection to the best response dynamics for potential games, domain decomposition for PDEs, optimal transport, and coupled wave equations. In the case of the Tikhonov approximation, a natural choice is to take $\Psi (x)=  \|x-x_d\|^{2}$  where $x_d$ is a desired state (which is also the continuous model of the Halpern method \cite{Suh23}). By doing so, we obtain asymptotically the closest possible solution to $x_ d$. By translation, we can immediately reduce ourselves to the case $x_d=0$, as considered in our work.

The following theorem establishes the convergence of the trajectories of \eqref{CSGDT1} towards the minimum norm solution under minimal assumptions on the parameter $\varepsilon (t)$. We assume that \eqref{CSGDT1} admits a unique strong global solution $x: [0,+\infty[ \to \H$, namely, $x$ is absolutely continuous on each compact interval such that \eqref{CSGDT1} holds for almost every $t \geq 0$. Sufficient conditions for this well-posedness may be found in \cite{brezis}.

\begin{theorem}\label{cps}
\textit{Suppose that $\varepsilon : [t_0, +\infty[ \to \R_+$ is a measurable function that satisfies}:
\begin{renumerate}
    \item \label{tikzero}$\varepsilon (t) \to 0$ \textit{as} $t \to +\infty$;
    \item \label{tikinf}$\displaystyle{\int_{t_0}^{+\infty} \varepsilon (t) dt = +\infty}$.
\end{renumerate}
\textit{Let $x(\cdot)$ be a solution trajectory of the continuous dynamic \eqref{CSGDT1}. Then, $\slim_{t \to +\infty} x(t) = x^{\star} \eqdef \proj_{\calS_F} (0)$}.
\end{theorem}

\tcb{
This result was established in \cite[Theorem~2]{CPS}. For the reader's convenience, we give a self-contained short proof in Appendix~\ref{sec:proofcps}.
}

\subsection{Stochastic differential equations}\label{sec:sdes}

As said before, in many cases, the drift term is subject to noise. In such a scenario, one can model these errors using a stochastic integral with respect to the measure defined by a continuous It\^o martingale. In the smooth case without Tikhonov regularization, this approach has been well documented in  Maulen-Soto,  Fadili, Attouch \cite{mio}. 
This concerns the following stochastic differential equation as the stochastic counterpart of the gradient flow, let $t_0\geq 0$ and initial data $X_0\in\Lp^{\nu}(\Omega;\H)$ (for some $\nu\geq 2$):
\begin{equation}\label{CSGD2}\tag{$\mathrm{SDE}$}
\begin{cases}
\begin{aligned}
dX(t)&=-\nabla f(X(t))dt+\sigma(t,X(t))dW(t), \quad t> t_0\\
X(t_0)&=X_0.
\end{aligned}
\end{cases}
\end{equation}
Let us make precise the ingredients of this stochastic differential equation.
It is defined over a filtered probability space $(\Omega,\mathcal F,\{\mathcal F_t\}_{t\geq 0},\mathcal P)$, where the diffusion (volatility) term $\sigma:[t_0,+\infty[\times\H\rightarrow \calL_2(\K;\H)$ is a measurable function, and $W$ is a $\K$-valued cylindrical Brownian motion. \\

Throughout this article, the diffusion term $\sigma$ is assumed to satisfy \eqref{H}. In connection with this assumption, let us define $\sigma_*>0$ and $\sigma_{\infty}(\cdot)$ by
\begin{equation}
\norm{\sigma(t,x)}_{\mathrm{HS}}^2\leq \sigma_*^2, \quad \forall t\geq 0,\forall x\in\H, \quad 
\sigma_{\infty}(t)\eqdef \sup_{x\in\H}\norm{\sigma(t,x)}_{\mathrm{HS}} ,
\end{equation}
and $\sigma_{\infty}(\cdot)$ is a decreasing function.

Concerning the study of \eqref{SDI0} and \eqref{CSGD}, let us recall the following result of \cite[Theorem~3.1]{mio} on which we will build our study. It establishes almost sure weak convergence of $X(t)$ to an $\calS$-valued random variable as $t \to +\infty$.

\begin{theorem}\label{converge2}
Consider the dynamic \eqref{CSGD2} where $f$ and $\sigma$ satisfy the assumptions \eqref{H0} and \eqref{H}. Let $\nu\geq 2$, and its initial data $X_0\in\Lp^{\nu}(\Omega;\H)$. Then, there exists a unique solution $X\in S_{\H}^{\nu}[t_0]$ of \eqref{CSGD2}.  Additionally, if $\sigma_{\infty}\in \mathbb \Lp^2([t_0,+\infty[)$, then: 
\begin{renumerate}
\item \label{acota-basic} $\sup_{t\geq 0}\EE[\norm{ X(t)}^2]<+\infty$.
\item  $\forall x^{\star}\in \calS$, $\lim_{t\rightarrow +\infty} \norm{ X(t)-x^{\star}}$ exists a.s. and $\sup_{t\geq 0}\norm{ X(t)}<+ \infty$ a.s.
\item \label{iiconv-basic} %If there exists $C_0>0$ such that $\norm{\nabla f(X(t))}\leq C_0$ a.s., then $\lim_{t\rightarrow +\infty} f(X(t))=\min f$ a.s. 
$\lim_{t\rightarrow \infty}\norm{\nabla f(X(t))}=0$ a.s.  
As a result, $\lim_{t\rightarrow \infty} f(X(t))=\min f$ a.s.
\item  There exists an $\calS$-valued random variable $X^{\star}$ such that $\wlim_{t\rightarrow +\infty} X(t) = X^{\star}$ a.s.
\end{renumerate}
\end{theorem}

\begin{remark}
To be precise, \cite[Theorem~3.1]{mio} treats the finite-dimensional case, however in \cite[Chapter 3]{mitesis} the general separable real Hilbertian case was considered.
\end{remark}
%%%%%SDI%%%%%%%%%%%%%%%

\section{Stochastic differential inclusions}\label{sec:sdi}

In this section, we will work with stochastic differential inclusions. For the history of this concept, we refer the reader to \cite[Preface]{sdia}. We will start by showing a general version of the \eqref{SDI0} dynamic, formally describing what it means to be a solution of that dynamic, and then we will move on to show the conditions under which \tcb{we} can have the existence and uniqueness of a solution. Existence is due to \cite{petterson} and uniqueness is proven here. Then we will focus on \eqref{SDI0} and study the conditions on the diffusion term in order to ensure the almost sure weak convergence of the trajectory towards the set of minimizers. Finally, we will show some convergence rates of the objective under convexity or strong convexity.

%\todo{Reorganize into subsections: Existence and uniqueness, Convergence of the trajectories, Convergence rates. Also many equations beyond page limit. Insert line breaks at all appropriate places.}

\subsection{Existence and uniqueness of solution}
For a set-valued operator $A: \H \rightrightarrows \H$, its domain is $\dom(A)= \{x \in \H:~ A(x) \neq \emptyset\}$. For $t_0 \geq 0$, let $b:[t_0,+\infty[\times \H\rightarrow\H$ and $\sigma:[t_0,+\infty[\times\H\rightarrow\calL_2(\K;\H)$, and consider the general stochastic differential inclusion:
\begin{equation}\label{SDI}\tag{$\mathrm{SDI_0}$}
\begin{cases}
\begin{aligned}
dX(t)&\in b(t,X(t))dt-A(X(t))dt+\sigma(t,X(t))dW(t), \quad t> 0\\
X(t_0)&=X_0,
\end{aligned}
\end{cases}
\end{equation}
defined over a complete filtered probability space $(\Omega,\calF,\{\calF_t\}_{t\geq t_0},\PP)$, where the diffusion (volatility) term $\sigma:[t_0,+\infty[\times\H\rightarrow \calL_2(\K;\H)$ is a measurable function; $W$ is a $\calF_t$-adapted $\K$-valued cylindrical Brownian motion; and the initial data $X_0$ is an $\calF_0$-measurable $\H$-valued random variable.

\begin{definition}\label{def:SDIsolution}
A solution of \eqref{SDI} is  a couple $(X,\eta)$ of $\calF_t$-adapted processes such that almost surely:
\begin{renumerate}
    \item $X$ is continuous with sample paths in the domain of $A$;
    \item $\eta$ is absolutely continuous, such that $\eta(t_0)=0$, and $\forall T>t_0$, $\eta'\in\Lp^2([t_0,T];\H)$,    $\eta'(t)\in A(X(t))$ for almost all $t\geq t_0$;
    \item For $t> t_0$,
   \begin{equation}\label{itosdis}
 \begin{aligned}
 \begin{cases}
     X(t)&=X_0+\int_{t_0}^t b(s,X(s))ds-\eta(t)+\int_{t_0}^t \sigma(s,X(s))dW(s),\\
      X(t_0)&=X_0.
      \end{cases}
 \end{aligned}
     \end{equation}
\end{renumerate}
\end{definition}

For the sake of brevity, we sometimes omit the process $\eta$ and say that $X$ is a solution of \eqref{SDI}, meaning that,  there exists a process $\eta$ such that $(X,\eta)$ satisfies the previous definition. 

\smallskip

The definition of uniqueness for the process $X$ will be presented in Section \ref{onstochastic}.\\

\noindent  Throughout the paper it will be assumed that:
\begin{equation*}
\tag{$\mathrm{H_0}(A)$}\label{H0a}
\begin{cases}
\text{$A$ is a maximal monotone operator with closed domain}; \\
\calS \eqdef A^{-1}(0)\neq\emptyset. 
\end{cases}
\end{equation*}  

\begin{equation*}
\tag{$\mathrm{H_0}(b,\sigma)$}\label{H0bs}
\begin{cases}
\exists L>0, 
\Vert b(t,x)-b(t,y)\Vert \vee \Vert\sigma(t,x)-\sigma(t,y)\Vert_{\mathrm{HS}}\leq L\Vert x-y\Vert, \forall t\geq t_0,\forall x,y\in\H;\\
\sup_{t\geq t_0}(\Vert b(t,0)\Vert \vee \Vert \sigma(t,0)\Vert_{\mathrm{HS}})<+\infty. 
\end{cases}
\end{equation*} 

The Lipschitz continuity assumption is mild and required to ensure the well-posedness of \eqref{SDI}.

\smallskip

We are interested in ensuring the existence and uniqueness of a solution for \eqref{SDI}. Although there are several works that deal with the subject of stochastic differential inclusions (see \cite{sdia,bocsan, prato, rascanu, petterson, govindan}), those of \cite{petterson,govindan} are the closest to our setting and define a solution in the sense of Definition~\ref{def:SDIsolution}, thus generalizing the work of Br\'ezis \cite{brezis} in the deterministic case to the stochastic setting. In this paper, we consider the sequence of solutions $\{X_{\lambda}\}_{\lambda>0}$ of the stochastic differential equations 
\begin{equation}\label{SDE}\tag{$\mathrm{SDE_{\lambda}}$}
\begin{cases}
\begin{aligned}
dX_{\lambda}(t)&= b(t,X(t))dt-A_{\lambda}(X_{\lambda}(t))dt+\sigma(t,X_{\lambda}(t))dW(t), \quad t> t_0\\
X_{\lambda}(t_0)&=X_0,
\end{aligned}
\end{cases}
\end{equation}
where $A_{\lambda}=(I-(I+\lambda A)^{-1})/\lambda$ is the Yosida approximation of $A$ with parameter $\lambda>0$. Under \eqref{H0a} and \eqref{H0bs}, as well as the integrability condition 
\begin{equation}\label{Hl}\tag{$\mathrm{H}_{\lambda}$}
\limsup_{\lambda\downarrow 0}\int_{t_0}^T \EE(\Vert A_{\lambda}(X_{\lambda}(t))\Vert^2)dt<+\infty, 
\end{equation} 
it was shown in \cite[Theorem~3.5]{petterson} that there exists a couple $(X,\eta)$ of stochastic processes such that for every $T>t_0$, 
\[
\lim_{\lambda\downarrow 0}\EE\left(\sup_{t\in [t_0,T]}\Vert X_{\lambda}(t)-X(t)\Vert^2\right)=0,\quad \lim_{\lambda\downarrow 0}\EE\left(\sup_{t\in [t_0,T]}\Vert \eta_{\lambda}-\eta\Vert^2\right)=0,
\]
where $\eta_{\lambda}(t)=\int_{t_0}^t A_{\lambda}(X_{\lambda}(s))ds,$ and that $(X,\eta)$ is a solution of \eqref{SDI} in the sense of Definition~\ref{def:SDIsolution}. Moreover, one can even have a.s. strong convergence of the process $X_{\lambda}$ to $X$ when the diffusion term is state-independent; see \cite[Proposition~6.3]{petterson}.
\begin{remark}\label{remarkhl}
In view of \cite[Proposition~3.1]{petterson} and Fatou's lemma, for \eqref{Hl} to hold, it is sufficient that the Yosida approximation $A_{\lambda}$ obeys a linear growth condition of the form $\limsup_{\lambda\downarrow 0}\norm{A_{\lambda}(x)} \leq C(1+\norm{x})$ for all $x \in \H$. For instance, in light of \cite[Proposition~23.43(i)]{BauschkeBookPrevious}, the last linear growth condition holds if $\dom(A)=\H$ and $\norm{A^0 (x)}\leq C(1+\norm{x})$ for all $x \in \H$, where $A^0(x)$ is the minimal norm element of $A(x)$. A typical case of interest in the setting of non-smooth optimization, which is at the heart of this paper, is when $A$ is the subdifferential of a globally Lipschitz continuous convex function. This includes a wealth of functions used in machine learning, statistics and data processing. Note that to control the linear growth of $A_{\lambda}(x)$ through that of $A^0 (x)$, \cite[Corollary~23.46(ii)]{BauschkeBookPrevious} tells us that the domain condition $\dom(A)=\H$ cannot be removed since $\norm{A_{\lambda}(x)}\uparrow+\infty$ as $\lambda\downarrow 0$ for $x \in \H \setminus \dom(A)$. We insist, however, on the fact that these conditions are only sufficient but not necessary and \eqref{Hl} can be verified beyond this case; see for instance the product structure of $A$ and $\sigma$ studied in \cite{Kree82} and specialized to normal cones in \cite[Remark~3.3(ii)]{petterson}.
%case where this condition holds is when  $A$ is full domain and there exists $C_0>0$ such that $\Vert A^0 (x)\Vert\leq C_0(1+\Vert x\Vert)$ for $x\in\H$, where $A^0(x)=\argmin_{y\in A(x)}\Vert y\Vert $. \tcb{However, if $A$ is not full domain, by \cite[Proposition~2.6]{brezis}, $\Vert A_{\lambda}(x)\Vert\uparrow+\infty$ as $\lambda\downarrow 0$ for $x\in\dom(A)\setminus \H$ and \eqref{Hl} would not hold.}
\end{remark}
\medskip
\tcb{Although the existence of a solution to \eqref{SDI} was established in \cite{petterson}, uniqueness was not addressed. In the following theorem, we provide conditions under which uniqueness also holds for such SDI. The proof is given in Subsection~\ref{exiuniappendix}.}
\begin{theorem}\label{exiuni}
Consider \eqref{SDI}, where $A$ and $(b,\sigma)$ satisfy the assumption \eqref{H0a} and \eqref{H0bs}, respectively. Additionally, suppose that $A$ satisfy \eqref{Hl} and let $\nu\geq 2$ such that $X_0\in\Lp^{\nu}(\Omega;\H)$ and is $\calF_0$-measurable. Then, there exists a unique solution $(X,\eta)\in S_{\H}^{\nu}[t_0]\times C^1([t_0,+\infty[;\H)$ of \eqref{SDI}. 
\end{theorem}

\begin{corollary}
Consider \eqref{SDE}, where $A$ and $(b,\sigma)$ satisfy the assumption \eqref{H0a} and \eqref{H0bs}, respectively. Additionally, let us consider that $A$ satisfy \eqref{Hl} and let $\nu\geq 2$ such that $X_0\in\Lp^{\nu}(\Omega;\H)$ and is $\calF_0$-measurable. Then, 
\[
\sup_{\lambda>0}\EE\left(\sup_{t\in [t_0,T]}\Vert X_{\lambda}(t)\Vert^{\nu}\right)<+\infty.
\]
\end{corollary}
\begin{proof}
    Since $A^{-1}(0)=A_{\lambda}^{-1}(0)$ and $A_{\lambda}$ is monotone, we replace $\eta'$ by $A_{\lambda}(X_\lambda)$ in the proof of Theorem \ref{exiuni}, then we realize that the constant that bounds $\EE\left(\sup_{t\in [t_0,T]}\Vert X_{\lambda}(t)\Vert^{\nu}\right)$ is independent from $\lambda$ to conclude. 
\end{proof}

%%%%%%%VERSION FUNCIONES%%%%%%%%%%%%%%

Let us present our extension of It\^o's formula for a multi-valued drift, which plays a central role in the study of SDIs.
\begin{proposition}\label{itos}
Consider \eqref{SDI} under the assumptions of Theorem \eqref{exiuni}. Let $(X,\eta)\in S_{\H}^{\nu}[t_0]\times C^1([t_0,+\infty[;\H)$ be the unique solution of \eqref{SDI}, and let $\phi: [t_0,+\infty[\times\H\rightarrow\R$ be such that $\phi(\cdot,x)\in C^1([t_0,+\infty[)$  for every $x\in\H$ and $\phi(t,\cdot)\in C^2(\H)$ for every $t\geq t_0$. Then the process $Y(t)=\phi(t,X(t)),$ is an It\^o Process such that for all $t\geq 0$
\begin{multline}
Y(t)=Y(t_0)+\int_{t_0}^{t} \frac{\partial \phi}{\partial t}(s,X(s)) ds
+\int_{t_0}^{t} \dotp{\nabla \phi(s,X(s))}{b(s,X(s))-\eta'(s)} ds\\
+\int_{t_0}^{t}\dotp{\sigma^{\star}(s,X(s))\nabla \phi(s,X(s))}{dW(s)}+\frac{1}{2}\int_{t_0}^{t} \tr[\sigma(s,X(s))\sigma^{\star}(s,X(s))\nabla^2\phi(s,X(s))]ds,
\end{multline}
where $\eta'(t)\in A(X(t))$ a.s. for almost all $t\geq t_0$. Moreover, if $\EE[Y(t_0)]<+\infty$, and if for all $T>t_0$ 
\[
\EE\pa{\int_{t_0}^T \norm{\sigma^{\star}(s,X(s))\nabla \phi(s,X(s))}^2 ds}<+\infty,
\]
then $\displaystyle\int_{t_0}^{t}\dotp{\sigma^{\star}(s,X(s))\nabla \phi(s,X(s))}{dW(s)}$ is a square-integrable continuous martingale and
\begin{multline}
\EE[Y(t)]=\EE[Y(t_0)]+\EE\pa{\int_{t_0}^{t} \frac{\partial \phi}{\partial t}(s,X(s)) ds}+\EE\pa{\int_{t_0}^{t} \dotp{\nabla \phi(s,X(s))}{b(s,X(s))-\eta'(s)} ds}\\
+\frac{1}{2}\EE\pa{\int_{t_0}^{t} \tr[\sigma(s,X(s))\sigma^{\star}(s,X(s))\nabla^2\phi(s,X(s))]ds}.
\end{multline}

\end{proposition}
\begin{proof}
The unique solution $(X,\eta)\in S_{\H}^{\nu}[t_0]\times C^1([t_0,+\infty[;\H)$ of \eqref{SDI} satisfies (by definition) the following equation:
 \begin{equation}\label{itosdis2}
 \begin{aligned}
 \begin{cases}
     X(t)&=X_0+\int_{t_0}^t [b(s,X(s))-\eta'(s)]ds+\int_{t_0}^t \sigma(s,X(s))dW(s),\quad t>t_0,\\
      X(t_0)&=X_0.
      \end{cases}
 \end{aligned}
     \end{equation} 
and $\eta'(s)\in A(X(s))$ for almost all $t\geq 0$ a.s.. Then, \eqref{itosdis2} is an It\^o process with drift $s\mapsto b(s,X(s))-\eta'(s)$ and diffusion $s\mapsto \sigma(s,X(s))$. Consequently, we can apply the classical It\^o's formula (see \cite[Section~2.3]{infinite}) to obtain the desired.
\end{proof}

\subsection{Almost sure weak convergence of the trajectory}
We consider $f+g$ (called the potential) and study the dynamic \eqref{SDI0} under the hypotheses \eqref{H0} (\ie $f \in C_L^{1,1}(\H) \cap \Gamma_0(\H), g\in \Gamma_0(\H)$) and \eqref{H}. Recall the definitions of $\sigma_*$ and $\sigma_{\infty}(t)$ from \eqref{eq:defsigstar}. Throughout the rest of the paper, we use the notation:
\begin{eqnarray*}
F(x)&\eqdef f(x)+g(x),\\
\Sigma(t,x) &\eqdef  \sigma(t,x)\sigma(t,x)^{\star}, \\
 \mathcal S_{F} & \eqdef  \argmin (F) .
\end{eqnarray*}

%Our first main result establish almost sure weak convergence of $X(t)$ to a random variable that takes values in $\calS_F$. It is based on It\^o's formula, and on Barbalat's and Opial's Lemma. 
%\tcb{Our first main result establishes almost sure weak convergence of $X(t)$ to a random variable taking values in $\mathcal{S}_F$. The proof relies on Lyapunov-type arguments and the use of Barbalat's Lemma and Opial's Lemma. While these tools are classical in deterministic settings, applying them in a stochastic framework requires particular care. In our case, this involves working almost surely and leveraging an It\^o-type formula compatible with our notion of solution (see Definition~\ref{def:SDIsolution} and Proposition~\ref{itos}). The necessary arguments and justifications are detailed in the proof. The arguments closely follows the proof approach of \cite[Theorem 3.1]{mio}, whose statement is restated in Theorem \ref{converge2}, and constitutes an extension of that result to the non-smooth setting. The main challenge in this setting lies in justifying It\^o's formula and handling the selection curve $\eta'(t) \in \partial g(X(t))$ a.s. The former is addressed by our specific notion of solution (see Definition~\ref{def:SDIsolution} and Proposition~\ref{itos}), while the latter is of a more technical nature and is treated in detail in the proof.
%}

\tcb{Our first main result establishes almost sure weak convergence of $X(t)$ to a random variable taking values in $\mathcal{S}_F$. The proof relies on Lyapunov-type arguments and the use of Barbalat's Lemma and Opial's Lemma. While these tools are classical in the deterministic setting, applying them in a stochastic framework requires particular care. In our case, this involves working almost surely and leveraging the properties of It\^o's formula and Robbins-Siegmund lemma (see Theorem~\ref{impp}). The overall strategy aligns with the approach used in \cite[Theorem~3.1]{mio} (restated in Theorem~\ref{converge2}), and extends that result to the non-smooth setting.}

\tcb{The main challenge in this extension lies in justifying It\^o's formula and handling the selection curve $\eta'(t) \in \partial g(X(t))$ a.s. The former is addressed through our specific notion of solution (see Definition~\ref{def:SDIsolution} and Proposition~\ref{itos}), while the latter is of a more technical nature and is treated in detail in the proof. We now state the result precisely.
}

\begin{theorem}\label{converge}
    Consider $F=f+g$ and $\sigma$ satisfying \eqref{H0} and \eqref{H} respectively. Suppose further that $\partial g$ verifies \eqref{Hl}. %be such that \begin{equation*}\label{Hsigma}\tag{$H_{\sigma}$}
    %    \begin{cases}
    %    \exists l_0>0 \Vert \sigma(t,x)-\sigma(t,y)\Vert_{\mathrm{HS}}\leq l_0\Vert x-y\Vert, \forall x,y\in \H.\\
    %    \sup_{t\geq t_0}\Vert \sigma(t,0)\Vert_{\mathrm{HS}}<+\infty.
    %\end{cases} \end{equation*} 
Let $\nu\geq 2$, $t_0\geq 0$ , and consider the dynamic \eqref{SDI0} with initial data $X_0\in \Lp^{\nu}(\Omega;\H)$, \ie:
\begin{equation}
    \begin{cases}
         dX(t)&\in -\partial F(X(t))dt+\sigma(t,X(t))dW(t),\\
        X(t_0)&=X_0,
    \end{cases}
\end{equation}
where $W$ is a $\K$-valued cylindrical Brownian motion. Then, there exists a unique solution (in the sense of Theorem \ref{exiuni}) $(X,\eta)\in S_{\H}^{\nu}[t_0]\times C^{1}([t_0,+\infty[;\H)$.

Moreover, if $\sigma_{\infty}\in\Lp^2([t_0,+\infty[)$, then the following holds:
    \begin{renumerate}
        \item \label{acota2fg} $\EE[\sup_{t\geq t_0}\norm{X(t)}^{\nu}]<+\infty$.
        \vspace{1mm}
        \item  $\forall x^{\star}\in \calS_F$, $\lim_{t\rightarrow +\infty} \norm{X(t)-x^{\star}}$ exists a.s. and $\sup_{t\geq t_0}\norm{ X(t)}<+ \infty$ a.s..
        \vspace{1mm}
        \item \label{iiconvfg} If $g$ is continuous, then $\forall x^{\star}\in \calS_F$, $\nabla f(x^{\star})$ is constant,  $\lim_{t\rightarrow \infty}\norm{\nabla f(X(t))-\nabla f(x^{\star})}=0$ a.s. for any $x^{\star}\in \calS_F$ ,and  \[\int_{t_0}^{+\infty} F(X(t))-\min F \hspace{0.1cm}dt<+\infty.\] 
        \vspace{1mm}
        \item If \ref{iiconvfg} holds, then there exists an $\calS_F$-valued random variable $X^{\star}$ such that $\wlim_{t\rightarrow +\infty} X(t)= X^{\star}$.
    \end{renumerate}    
\end{theorem}

\begin{remark}
Let us specialize the discussion of Remark~\ref{remarkhl}. By classical properties of the Yosida approximation, we know that
\[
(\partial g(x))_{\lambda}=\nabla g_{\lambda}(x)=\frac{1}{\lambda}(x-\mathrm{prox}_{\lambda g}(x)),
\] 
where $g_{\lambda}$ is the Moreau envelope of $g$ with parameter $\lambda>0$. If there exists $C>0$, such that for all $x \in \H$
\[
\Vert x-\mathrm{prox}_{\lambda g}(x)\Vert\leq \lambda C(1+\norm{x}),
\] 
then assumption \eqref{Hl} is satisfied by $\partial g$. This is the case if $g$ is a continuous convex function, \ie $\dom(\partial g)=\H$, and $\Vert\partial^0 g(x)\Vert\leq C(1+\Vert x\Vert)$ for all $x \in \H$. As mentioned in Remark~\ref{remarkhl}, if $g$ is convex and $L_g$-Lipschitz continuous, then $\sup_{x \in \H} \Vert\partial^0 g(x) \Vert \leq L_g$, and thus \eqref{Hl} is in force. These examples cover a wealth of functions encountered in practice such as in machine learning and signal processing. 
\end{remark}

\begin{proof}
\begin{renumerate}
    \item Directly from Theorem \ref{exiuni}.
       \item Since $F$ is convex, we first notice that $\calS_F=(\partial F)^{-1}(0)$.
       
       \smallskip
       
    Now let us consider $(X,\eta)\in S_{\H}^{\nu}[t_0]\times C^{1}([t_0,+\infty[;\H)$ be the unique solution of \eqref{SDI} given by Theorem \ref{exiuni}, and $\phi(x)=   \frac{\norm{x-x^{\star}}^2}{2}$, where $x^{\star}\in \calS_F$. Then by It\^o's formula
    \begin{align}
        \phi(X(t))&=\underbrace{\frac{\norm{X_0-x^{\star}}^2}{2}}_{\xi=\phi(X_0)}+\underbrace{\frac{1}{2}\int_{t_0}^t \tr[\Sigma(s,X(s))ds]}_{A_t}-\underbrace{\int_{t_0}^t \dotp{\eta'(s)+\nabla f(X(s))}{X(s)-x^{\star}} ds}_{U_t} \nonumber \\
        &+\underbrace{\int_{t_0}^t \dotp{\sigma^{\star}(s,X(s))\pa{X(s)-x^{\star}}}{dW(s)}}_{M_t}. 
        \label{basic_Ito11fg}
    \end{align}

      Let us observe that, since $\nu\geq 2$, we have that $\EE(\sup_{t\geq t_0}\norm{X(t)}^2)<+\infty$. Moreover, since $\sigma_{\infty}\in\Lp^2([t_0,+\infty[)$ we have
\[
\EE\pa{\int_{t_0}^{+\infty}\norm{\sigma^{\star}(s,X(s))\pa{X(s)-x^{\star}}}^2ds} \leq \EE\pa{\sup_{t\geq t_0}\norm{ X(t)-x^{\star}}^2}\int_{t_0}^{+\infty} \sigma_{\infty}^2(s)ds<+\infty.
\]
Therefore $M_t$ is a square-integrable continuous martingale. It is also  a continuous local martingale (see \cite[Theorem 1.3.3]{mao}), which implies that $\EE(M_t)=0$.

\smallskip

Moreover, since $F$ is a convex function, then $\partial F$ is a monotone operator. On the other hand $\eta'(t)\in \partial g(X(t))$ a.s. for almost all $t\geq t_0$, so \[\dotp{\eta'(t)+\nabla f(X(t))}{X(t)-x^{\star}}\geq 0, a.s. \text{for almost all $t\geq t_0$}.\] 

We have that $A_t$ and $U_t$ defined as in \eqref{basic_Ito11fg} are two continuously adapted increasing processes with $A_0=U_0=0$ a.s..
Since $\phi(X(t))$ is nonnegative and $\sup_{x\in\H}\norm{ \sigma(\cdot,x)}_{\mathrm{HS}}\in \Lp^2([t_0,+\infty[)$, we deduce that  $\lim_{t\rightarrow +\infty}A_t< +\infty$. Then, we can use Theorem \ref{impp} to conclude that
\begin{equation}\label{ecconvfg}
\int_{t_0}^{+\infty} \langle \eta'(t)+\nabla f(X(t)),X(t)-x^{\star}\rangle dt< +\infty \quad a.s.
\end{equation}
and
\begin{equation}\label{xkconvfg}
\forall x^{\star}\in \calS_F, \exists \Omega_{x^{\star}}\in\calF, \text{such that } \Pro(\Omega_{x^{\star}})=1 \text{ and } \lim_{t\rightarrow +\infty}\norm{ X(\omega,t)-x^{\star}} \text{ exists } \forall \omega\in \Omega_{x^{\star}}.
\end{equation}
Since $\H$ is separable, there exists a countable set $\calZ \subseteq \calS_F$, such that $\cl(\calZ)=\calS_F$ (where $\cl$ stands for the closure of the set). Let $\widetilde{\Omega}=\bigcap_{z\in \calZ}\Omega_z$. Since $\calZ$ is countable, a union bound shows
\[
\PP(\widetilde{\Omega})=1-\PP\pa{\bigcup_{z\in \calZ}\Omega_z^c}\geq 1-\sum_{z\in \calZ}\PP(\Omega_z^c) = 1.
\]
For arbitrary $x^{\star}\in \calS_F$, there exists a sequence $(z_k)_{k\in\N}\subseteq \calZ$ such that $\lim_{k\rightarrow\infty} z_k= x^{\star}$.
%Now let $\omega\in \widetilde{\Omega}$, 
In view of \eqref{xkconvfg}, for every $k\in\N$ there exists $\tau_k:\Omega_{z_k}\rightarrow\R_+$ such that
\begin{equation}\label{eq:taukomegafg}
\lim_{t\rightarrow +\infty}\norm{ X(\omega,t)-z_k}=\tau_k(\omega), \quad \forall\omega\in\Omega_{z_k}.
\end{equation}

Now, let $\omega \in \widetilde{\Omega}$. Since $\widetilde{\Omega} \subset \Omega_{z_k}$ for any $k \in \N$, and  using the triangle inequality and \eqref{eq:taukomegafg}, we obtain that 
\[
\tau_k(\omega) - \norm{z_k-x^{\star}} \leq \liminf_{t\rightarrow +\infty}\norm{X(\omega,t)-x^{\star}} \leq \limsup_{t\rightarrow +\infty}\norm{X(\omega,t)-x^{\star}} \leq \tau_k(\omega) + \norm{z_k-x^{\star}} .
\]
Now, passing to $k \to +\infty$, we deduce
\[
\limsup_{k\rightarrow +\infty}\tau_k(\omega) \leq \liminf_{t\rightarrow +\infty}\norm{X(\omega,t)-x^{\star}} \leq \limsup_{t\rightarrow +\infty}\norm{X(\omega,t)-x^{\star}} \leq \liminf_{k\rightarrow +\infty}\tau_k(\omega) ,
\]
whence we deduce that $\lim_{k\rightarrow +\infty}\tau_k(\omega)$ exists on the set $\widetilde{\Omega}$ of probability $1$, and in turn \[\lim_{t\rightarrow +\infty}\norm{X(\omega,t)-x^{\star}}=\lim_{k\rightarrow +\infty}\tau_k(\omega).\]

\smallskip
    
Let us recall that there exists $\Omega_{\rm{cont}}\in\calF$  such that $\Pro(\Omega_{\rm{cont}})=1$ and $X(\omega,\cdot)$ is continuous for every $\omega\in\Omega_{\rm{cont}}$. Now let $x^{\star}\in \calS_F$ arbitrary, since the limit exists, for every $\omega\in\widetilde{\Omega}\cap\Omega_{\rm{cont}}$ there exists $T(\omega)$ such that $\norm{ X(\omega,t)-x^{\star}}\leq 1+\lim_{k\rightarrow +\infty}\tau_k(\omega)$ for every $t\geq T(\omega)$. Besides, since $X(\omega,\cdot)$ is continuous, by Bolzano's theorem \[\sup_{t\in [0,T(\omega)]}\norm{ X(\omega,t)}=\max_{t\in [0,T(\omega)]}\norm{ X(\omega,t)} \eqdef h(\omega) <+\infty.\] Therefore,  $\sup_{t\geq t_0}\norm{ X(\omega,t)}\leq \max\{h(\omega),1+\lim_{k\rightarrow +\infty}\tau_k(\omega)+\norm{ x^{\star}}\}<+\infty$.
    
\smallskip
    
\item Let $N_t=\displaystyle\int_{t_0}^t \sigma(s,X(s))dW(s)$. This is a continuous martingale (w.r.t. the filtration $\calF_t$), which verifies
\[
\EE(\Vert N_t\Vert^2)=\EE\pa{\int_{t_0}^t \norm{\sigma(s,X(s))}_{\mathrm{HS}}^2ds}\leq \EE\pa{\int_{t_0}^{+\infty} \sigma_{\infty}^2(s)ds}<+\infty, \forall t\geq t_0.
\]
 According to Theorem \ref{convmartingale}, we deduce that there exists a $\H-$valued random variable $N_{\infty}$ w.r.t. $\calF_{\infty}$, and which verifies: $\EE(\Vert N_{\infty}\Vert^2)<+\infty$, and there exists $\Omega_N\in\calF$ such that $\Pro(\Omega_N)=1$ and 
\[
\lim_{t\rightarrow +\infty}N_t(\omega)= N_{\infty}(\omega) \mbox{ for every } \omega\in\Omega_N .
\]
On the other hand, since $x^{\star}\in (\partial F)^{-1}(0)=(\nabla f+\partial g)^{-1}(0)$, then $-\nabla f(x^{\star})\in \partial g(x^{\star})$. Let $T>t_0$ such that $\eta'(t)\in \partial g(X(t))$ a.s., consequently,

\begin{align*}
    \langle \eta'(t)+\nabla f(X(t)),X(t)-x^{\star}\rangle&=\underbrace{\langle \eta'(t)-(-\nabla f(x^{\star})),X(t)-x^{\star}\rangle}_{\geq 0}\\& +\langle \nabla f(X(t))-\nabla f(x^{\star}),X(t)-x^{\star}\rangle\\
    &\geq \langle \nabla f(X(t))-\nabla f(x^{\star}),X(t)-x^{\star}\rangle\\
    &\geq \frac{1}{L} \norm{\nabla f(X(t))-\nabla f(x^{\star})}^2,
\end{align*}
where $\langle \eta'(t)-(-\nabla f(x^{\star})),X(t)-x^{\star}\rangle\geq 0$ by monotonicity of $\partial g$. Then by \eqref{ecconvfg} we obtain \begin{equation}\label{integralbfg}
    \int_{t_0}^{+\infty}\norm{\nabla f(X(t))-\nabla f(x^{\star})}^2dt<+\infty \quad a.s.. 
\end{equation}

Let $\Omega_{\mathrm{HS}}\in\calF$ be the event where \eqref{ecconvfg} (and consequently \eqref{integralbfg}) is satisfied ( $\Pro(\Omega_{\mathrm{HS}})=1$). Let $\Omega_{\eta}\in\calF$ be the event where $\eta'(t)\in \partial g(X(t))$ for almost all $T>t_0$ ($\Pro(\Omega_{\eta})=1$). Finally, let $\Omega_{\mathrm{conv}} \eqdef \widetilde{\Omega}\cap\Omega_{\rm{cont}}\cap\Omega_{\mathrm{HS}}\cap\Omega_M\cap\Omega_{\eta}$, hence $\Pro(\Omega_{\mathrm{conv}})=1$. Let also $\omega\in\Omega_{\mathrm{conv}}\subseteq \Omega_{\mathrm{HS}}$ arbitrary, then \[\liminf_{t\rightarrow +\infty} \norm{ \nabla f(X(\omega,t))-\nabla f(x^{\star})}=0.\] If also \[\limsup_{t\rightarrow +\infty} \norm{ \nabla f(X(\omega,t))-\nabla f(x^{\star})}=0,\] then we conclude with the proof. Suppose by contradiction that there exists $\omega_0\in \Omega_{\mathrm{conv}}$ such that \[\limsup_{t\rightarrow +\infty} \norm{\nabla f(X(\omega_0,t))-\nabla f(x^{\star})}>0.\] Then, by Lemma~\ref{existenceof}, there exists $\delta(\omega_0)>0$ satisfying 
\[
0=\liminf_{t\rightarrow +\infty} \norm{\nabla f(X(\omega_0,t))-\nabla f(x^{\star})}<\delta(\omega_0)<\limsup_{t\rightarrow +\infty} \norm{\nabla f(X(\omega_0,t))-\nabla f(x^{\star})},
\] 
and there exists $(t_k)_{k\in\N}\subset \R_+$ such that $\lim_{k\rightarrow +\infty} t_k=+\infty$, 
\[ 
\norm{\nabla f(X(\omega_0,t_k))-\nabla f(x^{\star})}>\delta(\omega_0) \qandq t_{k+1}-t_k>1, \quad \forall k\in\N.
\]

Additionally, consider $\eta'(\omega_0,t)\in \partial g(X(\omega_0,t))$ for almost all $T>t_0$. Since $\sup_{t\geq t_0}\norm{X(\omega_0,t)}<+\infty$, $\partial g$ is full domain, and the fact that $\partial g$ maps bounded sets onto bounded sets, we have that there exists $C_{\eta}(\omega_0)\geq 0$ such that $\Vert\eta'(\omega_0,t)\Vert^2\leq C_{\eta}(\omega_0)$ for almost all $T>t_0$.

\smallskip

We allow ourselves the abuse of notation $X(t)\eqdef X(\omega_0,t), \eta'(t)\eqdef \eta'(\omega_0,t), C_{\eta}\eqdef C_{\eta}(\omega_0)$ and $\delta\eqdef \delta(\omega_0)$ during the rest of the proof from this point.
 
\smallskip

Let  \begin{itemize}
    \item $C_0\eqdef C_{\eta}+\Vert \nabla f(x^{\star})\Vert^2$;
    \item $C_1\eqdef\frac{(2C_0+1)^2-1}{C_0}>0$;
    \item $\varepsilon\in \left]0,\min\{\frac{\delta^2}{4L^2},C_1\}\right[$; \item and $C(\varepsilon)\eqdef\frac{\sqrt{C_0\varepsilon+1}-1}{4C_0}\in ]0,\frac{1}{2}]$.
    \end{itemize}
    Note that this choice entails that the intervals $\pa{[t_k,t_k+C(\varepsilon)]}_{k \in \N}$ are disjoint. On the other hand, according to the convergence property of $N_t$ and the fact that $\norm{\nabla f(X(t))-\nabla f(x^{\star})}\in \Lp^2([t_0,+\infty[)$, there exists $k'>0$ such that for every $k\geq k'$
\[
\sup_{t\geq t_k}\Vert N_t-N_{t_k}\Vert^2<\frac{\varepsilon}{4} \qandq \int_{t_k}^{+\infty} \Vert \nabla f(X(t))-\nabla f(x^{\star})\Vert^2dt<1.
\]
Also, we compute \begin{align*}
    \int_{t_k}^{t}\norm{\eta'(s)+\nabla f(X(s))}^2ds&\leq 2\int_{t_k}^{t}\norm{\nabla f(X(s))-\nabla f(x^{\star})}^2ds+2\int_{t_k}^{t}\norm{\eta'(s)+\nabla f(x^{\star})}^2ds\\
    &\leq 2+4C_0(t-t_k).
\end{align*}

Furthermore, $C(\varepsilon)$ was chosen such that $C(\varepsilon)+2C_0C(\varepsilon)^2\leq \frac{\varepsilon}{8}$. Besides for every $k\geq k'$, $t\in [t_k,t_k+C(\varepsilon)]$,
\begin{align*}
\norm{ X(t)-X(t_k)}^2&\leq 2(t-t_k)\int_{t_k}^t\norm{\eta'(s)+\nabla f(X(s))}^2 ds+2\Vert N_t-N_{t_k}\Vert^2\\
&\leq 4(t-t_k)+8C_0(t-t_k)^2+\frac{\varepsilon}{2}\leq \varepsilon.    
\end{align*}

%So $\norm{\nabla f(X(t))-\nabla f(X(t_k))}\leq C^{\star}\norm{ X(t)-X(t_k)}\leq C^{\star}\varepsilon\leq \frac{\delta}{2}$.\\
Since $\nabla f$ is $L$-Lipschitz and $L^2\varepsilon\leq \left(\frac{\delta}{2}\right)^2$ by assumption on $\varepsilon$, we have that for every $k\geq k'$ and $t\in [t_k,t_k+C(\varepsilon)]$
\[
\norm{ \nabla f(X(t))-\nabla f(X(t_k))}^2\leq L^2\norm{ X(t)-X(t_k)}^2\leq \pa{\frac{\delta}{2}}^2.
\]
Therefore, for every $k\geq k'$, $t\in [t_k,t_k+C(\varepsilon)]$
\[
\norm{\nabla f(X(t))-\nabla f(x^{\star})}\geq  \norm{ \nabla f(X(t_k))-\nabla f(x^{\star})}-\underbrace{\norm{ \nabla f(X(t))- \nabla f(X(t_k))}}_{\leq \frac{\delta}{2}}\geq \frac{\delta}{2}.
\]
Finally, 
\begin{align*}
\int_{t_0}^{+\infty} \norm{ \nabla f(X(s))-\nabla f(x^{\star})}^2 ds&\geq \sum_{k\geq k'}\int_{t_k}^{t_k+C(\varepsilon)} \norm{ \nabla f(X(s))-\nabla f(x^{\star})}^2 ds\\
&\geq\sum_{k\geq k'}\frac{\delta^2C(\varepsilon)}{4}=+\infty ,
\end{align*}
which contradicts $\norm{\nabla f(X(\cdot))-\nabla f(x^{\star})}\in \Lp^2([t_0,+\infty[)$. So, for every $\omega\in \Omega_{\mathrm{conv}}$, 
\begin{align*}
\limsup_{t\rightarrow +\infty} \norm{\nabla f(X(\omega,t))-\nabla f(x^{\star})}&=\liminf_{t\rightarrow +\infty} \norm{\nabla f(X(\omega, t))-\nabla f(x^{\star})} \\
&=\lim_{t\rightarrow +\infty} \norm{\nabla f(X(\omega,t))-\nabla f(x^{\star})}=0.
\end{align*}
%

%\item Let $\omega\in\Omega_{\mathrm{conv}}$ and $\widetilde{X}(\omega)$ be a sequential limit point of $X(\omega,t)$. Equivalently, there exists an increasing sequence $ (t_k)_{k\in\N}\subset [t_0,+\infty[$ such that $\lim_{k\rightarrow +\infty} t_k=+\infty$ and  
%\[
%\lim_{k\rightarrow +\infty} X(\omega, t_k) = \widetilde{X}(\omega).
%\]
%Since $\lim_{t\rightarrow +\infty} \norm{B(X(\omega,t))-B(x^{\star})}=0$ and by continuity of $B$, we obtain directly that $\widetilde{X}(\omega)\in (A+B)^{-1}(0)$. Finally, by Opial's Lemma (see \cite{opial}) we conclude that there exists $X^{\star}(\omega)\in (A+B)^{-1}(0)$ such that $\lim_{t\rightarrow +\infty}X(\omega,t)= X^{\star}(\omega)$. In other words, since $\omega\in\Omega_{\mathrm{conv}}$ was arbitrary, there exists an $(A+B)^{-1}(0)$-valued random variable $X^{\star}$ such that $\lim_{t\rightarrow +\infty} X(t)= X^{\star}$ a.s..
%
On the other hand, since $F$ is convex, by \eqref{ecconvfg}, we obtain \begin{equation}
    \int_{t_0}^{+\infty} F(X(t))-\min F \hspace{0.1cm}dt<+\infty,\quad a.s..
\end{equation}
Since $\sup_{t\geq t_0}\Vert X(t)\Vert<+\infty$ a.s., and $(\partial F)$ maps bounded sets onto bounded sets (since $g$ is convex and continuous), we can show that there exists $\tilde{L}>0$ such that \[ |F(X(t_1))-F(X(t_2))|\leq \tilde{L}\Vert X(t_1)-X(t_2)\Vert, \quad \forall t_1,t_2\geq t_0, a.s..\]

Using the same technique as before, we can conclude that $\lim_{t\rightarrow+\infty }F(X(t))=\min F$ a.s..

\item Let $\omega\in\Omega_{\mathrm{conv}}$ and $\widetilde{X}(\omega)$ be a weak sequential limit point of $X(\omega,t)$. Equivalently, there exists an increasing sequence $ (t_k)_{k\in\N}\subset \R_+$ such that $\lim_{k\rightarrow +\infty} t_k=+\infty$ and 
\[
\wlim_{k\rightarrow +\infty} X(\omega, t_k) = \widetilde{X}(\omega).
\]
Since $\lim_{t\rightarrow +\infty} F(X(\omega,t))=\min F$ and the fact that $F$ is weakly lower semicontinuous (since it is convex and continuous), we obtain directly that $\widetilde{X}(\omega)\in \calS_F$. Finally, by Opial's Lemma (see \cite{opial}) we conclude that there exists $X^{\star}(\omega)\in \calS_F$ such that \[\wlim_{t\rightarrow +\infty}X(\omega,t)= X^{\star}(\omega).\] In other words, since $\omega\in\Omega_{\mathrm{conv}}$ was arbitrary, there exists an $\calS_F$-valued random variable $X^{\star}$ such that $\wlim_{t\rightarrow +\infty} X(t)= X^{\star}$ a.s..
\end{renumerate}
\end{proof}

\subsection{Convergence rates of the objective}
The following result, stated below, summarizes the global convergence rates in expectation satisfied by the trajectories of \eqref{SDI0}, and it is a natural extension of \cite[Theorem 3.2]{mio} to the non-smooth setting.

\begin{theorem}\label{importante0}
Consider the dynamic \eqref{SDI0} where $F=f+g$ and $\sigma$ satisfy assumptions \eqref{H0} and \eqref{H}. Furthermore, assume that $\partial g$ satisfies \eqref{Hl} and that $X_0 \in \Lp^2(\Omega;\H)$ and is $\calF_0$-measurable. The following statements are satisfied by the unique solution trajectory $X \in S_{\H}^2[t_0]$ of \eqref{SDI0}:

\smallskip

\begin{renumerate}
\item \label{0i} Let $\displaystyle\overline{F\circ X}(t)\eqdef t^{-1}\int_{t_0}^t F(X(s))ds$ and $\displaystyle\overline{X}(t)=t^{-1}\int_{t_0}^t X(s)ds$. Then 
\begin{equation}\label{eq:0i}
\EE\pa{F(\overline{X}(t))-\min F}\leq \EE\pa{\overline{F\circ X}(t)-\min F}\leq \frac{\EE\pa{\dist(X_0,\calS_F)^2}}{2t}+\frac{\sigma_*^2}{2}, \quad \forall t> t_0.
\end{equation}
Besides, if $\sigma_{\infty}$ is $\Lp^2([t_0,+\infty[)$, then 
\begin{equation}
\EE\pa{F(\overline{X}(t))-\min F}\leq \EE\pa{\overline{F\circ X}(t)-\min F}=\calO\pa{\frac{1}{t}} .
\end{equation}

\item \label{0ii} Moreover, if $F\in \Gamma_{\mu}(\H)$ with $\mu > 0$, then $\calS_F=\{x^{\star}\}$ and 
       
\begin{equation}
\EE\pa{\norm{ X(t)-x^{\star}}^2}\leq \EE\pa{\norm{X_0-x^{\star}}^2}e^{-\mu t}+\frac{\sigma_*^2}{\mu}, \quad \forall t> t_0.
\end{equation}
Besides, if $\sigma_{\infty}$ is non-increasing and vanishes at infinity, then: 
\begin{equation}\label{conv_rate_strongly_convex}
\EE\pa{\norm{X(t)-x^{\star}}^2}\leq \EE\pa{\norm{X_0-x^{\star}}^2}e^{-\mu t}+\frac{\sigma_*^2}{\mu}e^{\frac{\mu t_0}{2}}e^{-\frac{\mu t}{2}}+\sigma_{\infty}^2\left(\frac{ t_0+t}{2}\right), \quad \forall t>t_0.
\end{equation}

\end{renumerate}

\end{theorem}
\begin{proof}
%Using that $F(x)-\min F\leq \langle y,x-x^{\star}\rangle$ for every $y\in\partial F(x), x^{\star}\in\calS_F$, and It\^o's formula with the anchor function $\phi(x)=\frac{\Vert x-x^{\star}\Vert^2}{2}$ (for $x^{\star}\in\calS_F$), the proof is analogous to \cite[Theorem 3.2]{mio}.

\tcb{By making use of the inequality
\begin{equation}\label{convexsf}
F(x) - \min F \leq \langle y, x - x^{\star} \rangle, \quad \text{for all } y \in \partial F(x) \text{ and } x^{\star} \in \mathcal{S}_F,    
\end{equation}
which is given by the convexity of $F$, and considering the anchor function defined by
\[
\phi(x) = \frac{1}{2} \|x - x^{\star}\|^2, \quad \text{for } x^{\star} \in \mathcal{S}_F,
\]
 we apply It\^o's formula to $\phi(X(t))$ and take expectation. Combining \eqref{convexsf} with the result of It\^o's formula leads to an integral equation governing the expected behavior of $\phi(X(t))$, which can subsequently be leveraged to establish convergence rates in expectation of the process. Since the arguments are essentially identical to those of \cite[Theorem~3.2]{mio}, we refrain from reproducing the details here.}

\end{proof}

\begin{remark}
\tcb{
   In the deterministic setting, \ie, \eqref{DI}, the quantity $F(\overline{x}(t))-\min F$ decreases monotonically and one directly obtains the standard $\mathcal{O}(1/t)$, besides when $F$ is strongly convex, we have linear convergence rate of the distance to the unique minimizer, these convergence rates are obtained defining the same Lyapunov functions as in Theorem \ref{importante0}, \ie, $\phi(x) = \frac{1}{2} \|x - x^{\star}\|^2$ for $x^{\star} \in \mathcal{S}_F$.  By contrast, for the stochastic differential inclusion \eqref{SDI0}, the main challenge we must overcome is that the term $\sigma(t,X(t))dW(t)$ generates both a martingale difference noise and a quadratic variation term. We note that the martingale difference noise term vanishes after taking expectation and that the quadratic variation term is controlled by uniform bounds on $\sigma$, leading to the bias terms $\tfrac{\sigma_*^2}{2}$ (and $\tfrac{\sigma_*^2}{\mu}$ under strong convexity), moreover when $\sigma_{\infty}$ is square integrable, the bias terms vanishes over time. These challenges, and the way they are addressed, are essential for extending the classical ODE-based analysis to the stochastic setting.
Naturally, setting $\sigma \equiv 0$ in \eqref{SDI0} eliminates the stochastic terms, and all the estimates in Theorem~\ref{importante0} reduce exactly to the classical bounds known for the deterministic subgradient flow.
}
\end{remark}

%%%%%%%%%%%%%%%%%%%TIKHONOV%%%%%%%%%%%%

\section{Tikhonov regularization: Convergence properties for convex functions}\label{sec:tikhonov}
%%%%%%%%%%%%%%%%%%%%%%%%%%%%%%%%%%%%%%%%%%%%%%%%%%%%%%%%%%%%%%%%%%%%%%
It is important to provide insight into the technique of Tikhonov regularization. This allows us to pass from the almost sure weak convergence towards the set of minimizers of the trajectory generated by \eqref{SDI} to achieving almost sure strong convergence of the trajectory generated by \eqref{CSGD}, not only towards the set of minimizers but to the minimal norm solution. The \tcb{trade-off} in order to achieve this is the proper tuning of the Tikhonov parameter that depends on a local constant that could be hard to compute, besides that, we obtain slower convergence rates of the objective, passing from $\mathcal{O}(t^{-1})$ to $\mathcal{O}(t^{-r}+R(t))$, where $r<1$ and $R(t)\rightarrow 0$ (defined below in \eqref{eqdefr}).

\smallskip

%%%%%%%%%%%%%%%%%%%%%%
\subsection{Almost sure convergence of the trajectory to the minimal norm solution}

Our second main result establish almost sure convergence of $X(t)$ to  $x^{\star}=\proj_{\calS_F}(0)$ as $t \to +\infty$. It is based on a subtle tuning of the Tikhonov parameter  $\varepsilon (t)$ formulated as conditions \ref{t1}, \ref{t2}, and \ref{t3}  below. We know that $ \|x^{\star} \|^2 -  \|x_{\varepsilon (t)} \|^2$ tends to zero as $t \to +\infty$. \tcb{By capitalizing on Proposition~\ref{ratetikhonov}, we shall see in Theorem \ref{practical} that the conditions \ref{t1}, \ref{t2}, and \ref{t3} are compatible for functions verifying H\"olderian-type error bounds, which is the case for {\L}ojasiewicz functions (see Definition~\ref{def:lojq} and Proposition~\ref{ebw})}.

\begin{theorem}\label{converge20}
Consider the dynamic \eqref{CSGD} where $F=f+g$ and $\sigma$ satisfy the assumptions \eqref{H0} and \eqref{H}, respectively, furthermore assume that $\partial g$ satisfy \eqref{Hl}. Let $\nu\geq 2$, and its initial data $X_0\in\Lp^{\nu}(\Omega;\H)$. Then, there exists a unique solution $X\in S_{\H}^{\nu}[t_0]$ of \eqref{CSGD}. Let $x^{\star}\eqdef\proj_{\calS_F}(0)$ be the minimum norm solution, and for $\varepsilon>0$ let $x_{\varepsilon}$ be the unique minimizer of $F_{\varepsilon}(x)\eqdef F(x)+\frac{\varepsilon}{2}\|x\|^2$. Suppose that 
$\sigma_{\infty}\in \Lp^2([t_0,+\infty[)$, and that $\varepsilon:[t_0,+\infty[\rightarrow\R_+$ satisfies the conditions:

\vspace{2mm}
\begin{Tenumerate}%[label=\rm{$(T_\arabic*)$}]
    \item \label{t1} $\varepsilon (t) \to 0$ \textit{as} $t \to +\infty$;%, and $\sup_{t\geq t_0}|\varepsilon(t)|<+\infty$;
    \item \label{t2} $\displaystyle{\int_{t_0}^{+\infty} \varepsilon (t) dt = +\infty}$;
    \item \label{t3} $\displaystyle{ \int_{t_0}^{+\infty} \varepsilon (t) \left(    \|x^{\star} \|^2 -  \|x_{\varepsilon (t)} \|^2  \right)dt  <+\infty}. $
\end{Tenumerate}

\vspace{2mm}

\noindent Then we have

\begin{renumerate}
\item \label{acota} 
$ \displaystyle\int_{t_0}^{+\infty} \varepsilon (t)  \EE[\|X(t) - x^{\star}\|^2]dt<+\infty$.
\vspace{1mm}
\item   $\lim_{t\rightarrow +\infty} \norm{ X(t)-x^{\star}}$ exists a.s. and $\sup_{t\geq t_0}\norm{ X(t)}<+ \infty$ a.s..
\item $\displaystyle\int_{t_0}^{+\infty} \varepsilon(t)\norm{ X(t)-x^{\star}}^2 dt<+\infty$ a.s..
\item $\slim_{t \to +\infty}X(t)=x^{\star}$ a.s.
\vspace{1mm}
 \end{renumerate}
\end{theorem}
\tcb{
\begin{remark}
We note that assumptions \ref{t1} and \ref{t2} are the same as those required in the deterministic setting, namely in Theorem~\ref{cps}. However, assumption~\ref{t3} is new and can be viewed as the price for moving to the stochastic case. This term appears in the proof of Theorem~\ref{cps}, specifically on the right-hand side of \eqref{xt3}. While we could recover Theorem \ref{cps} by directly assuming \ref{t3}, this is unnecessary in the deterministic case. In the latter, one deals with the differential inequality \eqref{xt3} to get the desired conclusion without relying on \ref{t3}. In contrast, in the stochastic setting, we obtain an integral inequality that cannot be handled in the same way, as the stochasticity removes some of the monotonicity structure inherent to the deterministic problem.
\end{remark}
}
\begin{proof}
The existence and uniqueness of a solution $X\in S_{\H}^{\nu}[t_0]$ follow directly from the fact that the conditions of Theorem~\ref{exiuni} are satisfied under \eqref{H0} and \eqref{H}. The only subtlety to check is that $\sup_{t\geq t_0}|\varepsilon(t)|<+\infty$, but this can be assumed without loss of generality since $\varepsilon(t)\rightarrow 0$ as $t\rightarrow+\infty$ (it might be necessary a redefinition of $t_0$). \smallskip
 
% Let $x^{\star}$ be the element of minimal norm of the closed convex nonempty set  $\calS$.  Setting $f_{\varepsilon} (x)\eqdef f(x) + \frac{\varepsilon}{2}\|x\|^2$ and letting $x_{\varepsilon}$ be the unique minimizer of $f_{\varepsilon}$. Then,
Our stochastic dynamic \eqref{CSGD} can be written equivalently as follows

\begin{equation}\label{CSGD4}\tag{$\mathrm{SDIT}$}
\begin{cases}
\begin{aligned}
dX(t)&\in -\partial F_{\varepsilon (t)}(X(t))dt  +\sigma(t,X(t))dW(t), \quad t\geq t_0;\\
X(t_0)&=X_0,
\end{aligned}
\end{cases}
\end{equation}

\begin{renumerate}
    \item 
Let us define the anchor function $\phi(x)=\frac{\norm{x-x^{\star}}^2}{2}$. Since $\partial g$ satisfy \eqref{Hl}, there exists a stochastic process $\tilde{\eta}:\Omega\times [t_0,+\infty[\rightarrow\H$ such that $\tilde{\eta}(t)\in\partial F_{\varepsilon (t)}(X(t))$ a.s. for almost all $t\geq t_0$.
Using It\^o's formula we obtain
\begin{align}
\phi(X(t))&=\underbrace{\frac{\norm{X_0-x^{\star}}^2}{2}}_{\xi}+\underbrace{\frac{1}{2}\int_{t_0}^t \tr[\Sigma(s,X(s))ds]}_{A_t}-\underbrace{\int_{t_0}^t \dotp{\tilde{\eta}(s)}{X(s)-x^{\star}} ds}_{U_t} \nonumber \\
&+\underbrace{\int_{t_0}^t \dotp{\sigma^{\star}(s,X(s))\pa{X(s)-x^{\star}}}{dW(s)}}_{M_t}. 
\label{basic_Ito1}
\end{align}
Since $X\in S_{\H}^2[t_0]$ by Proposition~\ref{itos}, we have for every $T>t_0$, that 
\[
\EE\pa{\int_{t_0}^T\norm{\sigma^{\star}(s,X(s))\pa{X(s)-x^{\star}}}^2ds} \leq \EE\pa{\sup_{t\in [t_0,T]}\norm{ X(t)-x^{\star}}^2}\int_{t_0}^{+\infty} \sigma_{\infty}^2(s)ds<+\infty.
\]
Therefore $M_t$ is a square-integrable continuous martingale. It is also a continuous local martingale, which implies that $\EE(M_t)=0$.

Let us now take the expectation of \eqref{basic_Ito1}. Using that 
\[
0\leq\tr[\Sigma(s,X(s))] \leq \sigma_{\infty}^2(s), 
\] 
and \eqref{dert-2} that we recall below
\begin{equation}\label{dert-2-b}
\left\langle  y(t),  X(t)- x^{\star} \right\rangle \geq   \varepsilon (t)\phi(X(t))  + \frac{\varepsilon (t)}{2} \left(    \|x_{\varepsilon (t)} \|^2 -  \|x^{\star} \|^2 \right),
\end{equation}
where $y:\Omega\times [t_0,+\infty[\rightarrow \H$ is such that $y(t)\in \partial F_{\varepsilon (t)} (X(t))$ a.s.. We obtain that
\begin{eqnarray*}
&&\EE\pa{\phi(X(t))}+ \int_{t_0}^{t} \varepsilon (s) \EE\pa{\phi(X(s))} ds
\\
&& \leq \EE\pa{\frac{\norm{ X_0-x^{\star}}^2}{2}}+\frac{1}{2}\int_{t_0}^{t} \sigma_{\infty}^2(s)ds + \frac{1}{2}\int_{t_0}^{t} \varepsilon (s) \left(    \|x^{\star} \|^2 -  \|x_{\varepsilon (s)} \|^2  \right)ds.
\end{eqnarray*}
According to our assumptions, we can write briefly the above relation as 
\begin{equation}\label{SDE-basic-T1}
\EE\pa{\phi(X(t))}+ \int_{t_0}^{t} \varepsilon (s) \EE\pa{\phi(X(s))} ds\leq \tcb{\Upsilon}(t),
\end{equation}
with $\tcb{\Upsilon}$ a nonnegative function defined by
\[\tcb{\Upsilon}(t) \eqdef \EE\pa{\frac{\norm{ X_0-x^{\star}}^2}{2}}+\frac{1}{2}\int_{t_0}^{t} \sigma_{\infty}^2(s)ds + \frac{1}{2}\int_{t_0}^{t} \varepsilon (s) \left(    \|x^{\star} \|^2 -  \|x_{\varepsilon (s)} \|^2  \right)ds
\]
which satisfies $\lim_{t\to +\infty} \tcb{\Upsilon}(t)= \tcb{\Upsilon}_{\infty} <+\infty$ by the fact that $X_0\in \Lp^2(\Omega;\H)$, $\sigma_{\infty}\in\Lp^2([t_0,+\infty[)$ and \ref{t3}.\\

Let us integrate the above relation \eqref{SDE-basic-T1}. We set
\[
\theta (t) \eqdef \int_{t_0}^{t} \EE\pa{\phi(X(s))} ds.
\]
We have $\dot{\theta}(t)=  \EE\pa{\phi(X(t))} $ and \eqref{SDE-basic-T1} is written equivalently as 
\begin{equation}\label{SDE-basic-T2}
\dot{\theta}(t)+ \int_{t_0}^{t} \varepsilon (s) \dot{\theta}(s) ds\leq \tcb{\Upsilon}(t).
\end{equation}
Equivalently
\begin{equation}\label{SDE-basic-T3}
\frac{1}{\varepsilon (t)}\dfrac{d}{dt}  \int_{t_0}^{t} \varepsilon (s) \dot{\theta}(s) ds+ \int_{t_0}^{t} \varepsilon (s) \dot{\theta}(s) ds\leq \tcb{\Upsilon}(t),
\end{equation} 
that is 
\begin{equation}\label{SDE-basic-T4}
\dfrac{d}{dt}  \int_{t_0}^{t} \varepsilon (s) \dot{\theta}(s) ds+ \varepsilon (t) \int_{t_0}^{t} \varepsilon (s) \dot{\theta}(s) ds\leq \varepsilon (t)\tcb{\Upsilon}(t).
\end{equation}
With $m(t) \eqdef \exp{\int_{t_0}^t \varepsilon (s)ds}$, we get
\begin{equation}\label{SDE-basic-T5}
\dfrac{d}{dt} \left( m(t) \int_{t_0}^{t} \varepsilon (s) \dot{\theta}(s) ds\right)  \leq \varepsilon (t) m(t) \tcb{\Upsilon}(t).
\end{equation}
After integration we get
\begin{equation}\label{SDE-basic-T6}
\int_{t_0}^{t} \varepsilon (s) \dot{\theta}(s) ds  \leq \frac{1}{m(t)}
 \int_{t_0}^{t}  m'(s) \tcb{\Upsilon}(s)ds.
\end{equation}
Since $\tcb{\Upsilon}$ is bounded by assumption \ref{t3}, we get

\[\sup_{t\geq t_0}\EE \Big[  \displaystyle{\int_{t_0}^{t} \varepsilon (s)  \|X(s) - x^{\star}\|^2}\Big]ds <+\infty.\] 
Equivalently
\[\int_{t_0}^{+\infty }  \EE \Big[ \|X(t) - x^{\star}\|^2\Big] \varepsilon (t) dt <+\infty .
\]
The assumption \ref{t2}
%$ \displaystyle{\int_0^{+\infty} \varepsilon (t) dt = +\infty}$
guarantees that the above inequality forces 
$\EE \Big[ \|X(t) - x^{\star}\|^2\Big]$ to tend to zero. 

\item Consider \eqref{basic_Ito1}, we define \begin{eqnarray*}
    \tilde{A}_t\eqdef A_t+ \int_{t_0}^t \frac{\varepsilon(s)}{2}(\| x^{\star}\|^2-\|x_{\varepsilon(s)}\|^2) ds, \quad
    and \quad \tilde{U}_t\eqdef U_t+\int_{t_0}^t \frac{\varepsilon(s)}{2}(\| x^{\star}\|^2-\|x_{\varepsilon(s)}\|^2)ds.
\end{eqnarray*}
By \eqref{dert-2-b} we have that $\tilde{U}_t\geq \int_{t_0}^{t} \varepsilon(s)\phi(X(s))ds\geq 0$. We can rewrite \eqref{basic_Ito1} as 
\[\phi(X(t))=\xi+\tilde{A}_t-\tilde{U}_t+M_t.\]
Since $\sigma_{\infty}\in\Lp^2([t_0,+\infty[)$ and \ref{t3}, then $\lim_{t\rightarrow +\infty}\tilde{A}_t<+\infty$. Let us observe that, since $X\in S_{\H}^2[t_0]$ by Proposition~\ref{itos}, we have for every $T>t_0$ that 
\[
\EE\pa{\int_{t_0}^T\norm{\sigma^{\star}(s,X(s))\pa{X(s)-x^{\star}}}^2ds} \leq \EE\pa{\sup_{t\in [t_0,T]}\norm{ X(t)-x^{\star}}^2}\int_{t_0}^{+\infty} \sigma_{\infty}^2(s)ds<+\infty.
\]
Therefore $M_t$ is a square-integrable continuous martingale. It is also  a continuous local martingale (see \cite[Theorem 1.3.3]{mao}), which implies that $\EE(M_t)=0$. 

By Theorem \ref{impp}, we get that $\lim_{t\rightarrow +\infty} \Vert X(t)-x^{\star}\Vert$ exists a.s. and that $\lim_{t\rightarrow +\infty} \tilde{U}_t<+\infty$ a.s..

\item Using the lower bound we had on $\tilde{U}_t$, we obtain \[\int_{t_0}^{+\infty} \varepsilon(t)\norm{X(t)-x^{\star}}^2dt<+\infty.\]

\item By the previous item, \ref{t2}, and Lemma \ref{lim0} we conclude that $\lim_{t\rightarrow +\infty} X(t)=x^{\star}$ a.s..
\end{renumerate}
This completes the proof.
\end{proof}

\subsection{Practical situations}

We will consider situations where the three conditions \ref{t1}, \ref{t2} and \ref{t3} are satisfied simultaneously. These are properties of the viscosity curve that we will now study.
The difficulty comes from \ref{t2} and \ref{t3} which are a priori not compatible. Indeed, \ref{t2} requires the parameter $\varepsilon (t)$ to converge slowly towards zero for the Tikhonov regularization to be effective. On the other hand in \ref{t3} the parameter $\varepsilon (t)$ must converge sufficiently quickly towards zero so that the term
$\left( \|x^{\star} \|^2 - \|x_{\varepsilon (t)} \|^2 \right)$ converges to zero fairly quickly, and thus corrects the infinite value of the integral of $ \varepsilon(t)$.

\subsubsection{{\L}ojasiewicz property}
Our first objective is to evaluate the rate of convergence towards zero of
$\left( \|x^{\star} \|^2 - \|x_{\varepsilon} \|^2 \right)$ as $\varepsilon \to 0$.
Using the differentiability properties of the viscosity curve is not a good idea, because the viscosity curve can be of infinite length in the case of a general differentiable convex function, see \cite{torralba}. To overcome this difficulty, we assume that $F=f+g$ satisfies the {\L}ojasiewicz property. This basic property has its roots in algebraic geometry, and it essentially captures a domination inequality between the objective value and its (sub)gradient. 
\begin{definition}[{\L}ojasiewicz inequality]\label{def:lojq}
Let $f:\H\rightarrow\R$ be a convex function with $\calS\neq\emptyset$ and $q\in [0,1[$. $f$ satisfies the {\L}ojasiewicz inequality on $\calS$ with exponent $q$ if there exists $r>\min f$ and $\mu > 0$ such that:
\begin{equation}\label{li}
\mu(f(x)-\min f)^{q}\leq \norm{\partial^{0}  f(x)},\quad \forall x \in  [\min f < f < r] ,
\end{equation}
%where $\partial^{0}f(x)=\argmin_{y\in\partial f(x)}\Vert y\Vert$.
%The function $f$ has the {\L}ojasiewicz property on $\calS$ if it obeys \eqref{li} at each point of $\calS$ with the same constant $\mu$ and exponent $q$, 
where we recall $\Vert\partial^{0}  f(x)\Vert=\min_{u\in\partial f(x)}\Vert u\Vert$. We will write $f \in \Loj^q(\calS)$. 
\end{definition}

Error bounds have also been successfully applied to various branches of optimization, and in particular to complexity analysis. Of particular interest in our setting is the H\"olderian error bound.
\begin{definition}[H\"olderian error bound]
Let $f:\H\rightarrow\R$ be a proper function such that $\calS\neq \emptyset$. Then $f$ satisfies a H\"olderian (or power-type) error bound inequality on $\calS$ with exponent $p \geq 1$, if there exists $\gamma>0$ and $r>\min f$ such that:
\begin{equation}\label{eq:errbnd}
f(x)-\min f \geq \gamma\dist(x,\calS)^p, \quad \forall x \in [\min f \leq f \leq r], 
\end{equation}
%For a given $r>\min f$ such that \eqref{eq:errbnd} holds, we will use the shorthand notation $f\in \EB^p(\calS)([f \leq r])$.
and we will write $f \in \EB^p(\calS)$.
\end{definition}

A deep result due to {\L}ojasiewicz states that for arbitrary continuous semi-algebraic functions, the H\"olderian error bound inequality holds on any compact set, and the {\L}ojasiewicz inequality holds at each point. In fact, for convex functions, the {\L}ojasiewicz property and H\"olderian error bound are actually equivalent. 

\begin{proposition}\label{ebw} 
Assume that $f\in\Gamma_0(\H)$ with $\calS\neq \emptyset$. Let $q\in[0,1[$, $p\eqdef \frac{1}{1-q}\geq 1$ and $r > \min f$. Then $f$ verifies the {\L}ojasiewicz inequality with exponent $q$ (see \eqref{li}) at $[\min f < f < r]$ if and only if the H\"olderian error bound with exponent $p$ (see \eqref{eq:errbnd}) holds on $ [\min f < f < r]$.
%\widetilde{\mu}_r\eqdef \pa{\frac{\mu}{p}}^p$. 
%, then there exists $\widetilde{\mu}_{r}>0$ such that $f\in \EB^p([f\leq r])$, \ie:  \begin{equation}\label{zz}
%    f(x)-\min f\geq \widetilde{\mu}_{r}\dist(x,\calS)^p,\quad \forall x\in [f\leq r].
%\end{equation}
%Moreover, if $f\in q(H)$, then (\ref{zz}) holds for every $x\in H$ with a unique $\mu>0$.
\end{proposition}
\begin{proof}
\tcb{This is a consequence of \cite[Theorem~5]{bolte}.}
\end{proof}

%\subsubsection{Resolvent interpretation of the viscosity curve}

\subsubsection{Quantitative stability of variational systems}

Our first objective is to evaluate the rate of convergence towards zero of $\left( \|x^{\star} \|^2 - \|x_{\varepsilon} \|^2 \right)$ as $\varepsilon \to 0$.
Using the differentiability properties of the viscosity curve is not a good idea, because the viscosity curve can be of infinite length in the case of a general differentiable convex function, see \cite{torralba}. To overcome this difficulty, we assume that $F=f+g$ satisfies the {\L}ojasiewicz property (see \eqref{li}). This basic property has its roots in algebraic geometry, and it essentially describes a relationship between the objective value and its gradient (or subgradient). Once this is assumed, we will need tools from variational analysis to conclude.

We start by recalling the notion of bounded Hausdorff distance for functions introduced in \cite{AW1} to study stability of minimization problems. All the results of this section until Theorem~\ref{practical} do not need separability of $\H$. 

For a set $C \subset \H \times \R$ and $\rho \geq 0$, we denote
$
C_\rho \eqdef C \cap \rho \Ball ,
$
where $\Ball$ is the unit ball in the box norm on $\H \times \R$. 
For two sets $C, D \subset \H \times \R$, the excess function of $C$ on $D$ is defined as
\[
e(C, D) \eqdef \sup_{x \in C} \dist(x, D) .
\]
For any $\rho \geq 0$, the $\rho$-Hausdorff distance between $C$ and $D$ is defined as
\[
\haus_\rho(C, D) \eqdef \max(e(C_\rho,D), e(D_\rho,C)).
\]
For $\rho=+\infty$, we recover the Hausdorff distance. A metrizable topology is naturally attached to the $\rho$-Hausdorff distance. When $\H$ is finite dimensional, the convergence with respect to the $\rho$-Hausdorff distances is nothing but the classical Painlev\'e-Kuratowski set-convergence. 

\begin{definition}\label{def:rhohausfun}
For $\rho \geq 0$, the $\rho$-Hausdorff (epi-)distance between two functions $f, g: \H \to \Rinf$ is
\[
\haus_\rho(f, g) \eqdef \haus_\rho(\epi f, \epi g) .
\]
\end{definition}
This device was extended in \cite{AMR1} to set-valued operators by identifying them with their graphs.
\begin{definition}\label{def:rhohausop}
For $\rho \geq 0$, the $\rho$-Hausdorff distance between two operators $A, B: \H \rightrightarrows \H$ is
\[
\haus_\rho(A, B) \eqdef \haus_\rho(\gra A, \gra B),
\]
where the unit ball is that of $\H \times \H$ equipped with the box norm.
\end{definition}

We recall the following two results that have been obtained in \cite{AMR1, AW1} and that will be important to prove our quantitative stability result. \tcb{It is a particular case of \cite[Proposition~1.2]{AMR1} with $\rho=0,\lambda=1$}.
\begin{proposition}\label{resolhaus}
    Let $A,B:\H \rightrightarrows \H$ be two maximal monotone operators, \tcb{then 
    \[
    \Vert J_A(0)- J_B(0)\Vert\leq 3\haus_{\Vert J_A(0)\Vert}(A,B),
    \]
    where $J_A\eqdef (I+A)^{-1},J_B\eqdef (I+B)^{-1}$ are the resolvent of the operators $A$ and $B$, respectively.}
\end{proposition}

The second abstract result is the equivalence of the uniform structure on the class of subdifferentials of convex lsc functions between the bounded Hausdorff distance and the uniform convergence on bounded sets of resolvents. 
\begin{proposition}[\tcb{\cite[Theorem~5.2]{AW1}}].\label{gradfun} Let $f$ and $g \in \Gamma_0(\H)$. To any $\rho > \max\pa{\dist(0, \mathrm{epi}(f)), \dist(0, \mathrm{epi}(g))}$ there correspond some constants $\kappa$ and $\rho_0$ (that depend on $\rho$) such that
\[
\haus_\rho (\partial f, \partial g) \leq  \kappa\pa{\haus_{\rho_0}(f, g)}^{\frac{1}{2}}.
\]
\end{proposition}

The following proposition is new and is a consequence of the previous two results. Since this is not obvious, we are going to present the whole proof.
\begin{proposition}\label{ratetikhonov}
Let $f\in\Gamma_0(\H)$ be a function such that $\calS \eqdef \argmin_{\H}(f)\neq\emptyset$, and that $f\in\EB^p(\calS)$. Let also $x^{\star}=\proj_{\calS}(0)$ and for $\varepsilon>0$, let $x_{\varepsilon}$ be the unique minimizer of $f_{\varepsilon}(x)=f(x)+\frac{\varepsilon}{2}\Vert x\Vert^2$. Then there exists $C_0,\varepsilon^{\star}>0$ such that 
    \begin{equation}\label{eq:ratetikhonov}
    \Vert x_{\varepsilon}-x^{\star}\Vert\leq C_0\varepsilon^{\frac{1}{2p}}, \quad\forall \varepsilon\in ]0,\varepsilon^{\star}].
    \end{equation}
    Consequently, there exists $C>0$ such that 
    \begin{equation}\label{eq:normratetikhonov}
    \Vert x^{\star}\Vert^2-\Vert x_{\varepsilon}\Vert^2\leq C\varepsilon^{\frac{1}{2p}}, \quad\forall \varepsilon\in ]0,\varepsilon^{\star}].
    \end{equation}
    \end{proposition}

\begin{proof}
\tcb{Let $\varphi_{\varepsilon}    \eqdef  \frac{1}{\varepsilon} \left( f- \min f    \right)$. By optimality of $x_{\varepsilon}$, we have
%\[
%x_{\varepsilon} + \frac{1}{\varepsilon} \partial f(x_{\varepsilon}) \ni 0,
%\]
%that is
\[
x_{\varepsilon} = \left( I + \partial \varphi_{\varepsilon}    \right)^{-1} (0) {=J_{\partial \varphi_{\varepsilon}}(0)} .
\]
}
We have that $\varphi_{\varepsilon}$ increases to \tcb{$\iota_{\calS}$} as $\varepsilon$ decreases to zero, and 
\[
x^{\star}=\proj_{\calS}(0) =  \left( I + \partial\iota_{\calS} \right)^{-1} (0)\tcb{=J_{\partial\iota_{\calS}}(0)} ,
\]
where \tcb{$\iota_{\calS}:\H\rightarrow \{0,+\infty\}$ is the indicator function of $\calS$, that takes $0$ on $\calS$ and $+\infty$ otherwise}. Therefore
\[
\|x_{\varepsilon} -  x^{\star} \|=  \|\left( I + \partial \varphi_{\varepsilon}    \right)^{-1} (0) -  \left( I + \partial\iota_{\calS} \right)^{-1} (0) \|\tcb{=\Vert J_{\partial \varphi_{\varepsilon}}(0)-J_{\partial\iota_{\calS}}(0)\Vert}.
\]
\tcb{Applying Proposition~\ref{resolhaus}} with $A=\partial \varphi_{\varepsilon}$, and $B=\partial \iota_{\calS}$, we have that 
\[
\|x_{\varepsilon} -  x^{\star} \|\leq 3\haus_{\rho}(\partial \varphi_{\varepsilon},\partial \iota_{\calS}),
\]
for $\rho>\Vert x^{\star}\Vert$. Now, since $\max\pa{\dist(0,\mathrm{epi}(\varphi_{\varepsilon})),\dist(0,\mathrm{epi}(\iota_{\calS}))}\leq \Vert x^{\star}\Vert$, we fix $\rho \geq \norm{x^\star}$, and Proposition~\ref{gradfun} entails that there exists constants $\kappa,\rho_0>0$ (depending on $\rho$) such that 
\begin{equation}\label{eq:xepstohausfun}
\|x_{\varepsilon} -  x^{\star} \|\leq 3\kappa\pa{\haus_{\rho_0}(\varphi_{\varepsilon},\iota_{\calS})}^{\frac{1}{2}}.
\end{equation}

To complete our proof we just need to bound the right hand side of the last inequality.
%\[
%\haus_{\rho_0} (\varphi_{\varepsilon}, \iota_{\calS}).
%\]
Observe first that since $\iota_{\calS} \geq \varphi_{\varepsilon}$ we just need to compute $e\pa{(\epi \varphi_{\varepsilon})_{\rho_0},\epi \iota_{\calS}} = e\pa{(\epi \varphi_{\varepsilon})_{\rho_0},\calS \times \R_+}$. 
%We have 
%\[
%\iota_{\calS}(x) \geq \varphi_{\varepsilon}(x) \geq \frac{\gamma}{\varepsilon}\dist(x,\calS)^p,\quad\forall x\in [\min f\leq f\leq r].
%\]
%Elementary computation gives 
It then follows from Definition~\ref{def:rhohausfun} that
\[
\haus_{\rho_0} (\varphi_{\varepsilon}, \iota_{\calS})=\max_{(x_1,r_1)\in\mathrm{epi}(\varphi_{\varepsilon})\cap \rho_0\mathbb{B}}\min_{(x_2,r_2)\in \calS \times \R_+)} \max\pa{\norm{x_1-x_2},|r_1-r_2|} ,
\]
where $\mathbb{B}$ is the unit ball of the max norm on $\H \times \R_+$. Besides, the inner minimization problem is bounded above by taking $r_2 = r_1$. Hence,  
\begin{align*}
\haus_{\rho_0} (\varphi_{\varepsilon}, \iota_{\calS}) 
&\leq \max_{(x_1,r_1)\in\mathrm{epi}(\varphi_{\varepsilon})\cap \rho_0\mathbb{B}}\dist(x_1,\calS) = \max_{x_1 \in [\min f \leq f \leq \varepsilon r_1 + \min f], \norm{x_1} \leq \rho_0, r_1 \leq \rho_0}\dist(x_1,\calS) \\
& \leq \max_{x \in [\min f \leq f \leq \varepsilon \rho_0 + \min f], \norm{x} \leq \rho_0}\dist(x,\calS) .
\end{align*}
We will now invoke the assumption that $f \in \EB^p(\calS)$. By the latter, there exists $\gamma>0,r>\min f$ such that \eqref{eq:errbnd} holds. Now choose $\varepsilon_0\eqdef  \frac{r-\min f}{\rho_0}>0$. We then have for any $\varepsilon\in ]0,\varepsilon_0]$ that
\begin{equation}\label{eq:hausfunbnd}
\haus_{\rho_0} (\varphi_{\varepsilon}, \iota_{\calS}) 
\leq \max_{x \in [0 \leq f - \min f \leq \varepsilon \rho_0]} \dist(x,\calS) \leq \max_{x \in [0 \leq f - \min f \leq \varepsilon \rho_0]} \pa{\frac{f(x) - \min f}{\gamma}}^{\frac{1}{p}} \leq \left(\frac{\rho_0}{\gamma}\right)^{\frac{1}{p}}\varepsilon^{\frac{1}{p}},
\end{equation}
where we have used that $\rho_0\varepsilon \leq \rho_0\varepsilon_0 \leq r - \min f$ so that \eqref{eq:errbnd} applies. Combining \eqref{eq:xepstohausfun} and \eqref{eq:hausfunbnd} gives \eqref{eq:ratetikhonov} where $C_0 = 3\kappa\left(\frac{\rho_0}{\gamma}\right)^{\frac{1}{2p}}$.%, then for $\varepsilon\in ]0,\varepsilon_0]$,  \[\Vert x_{\varepsilon}-x^{\star}\Vert\leq C_0\varepsilon^{\frac{1}{2p}}.\]

On the other hand, from Theorem~\ref{thm:hierarmin}(i) and the triangle inequality, we have
%since $\lim_{\varepsilon\rightarrow 0^+}x_{\varepsilon}= x^{\star}$, there exists $\varepsilon_1>0$ such that $\Vert x_{\varepsilon}\Vert\leq 1+\Vert x^{\star}\Vert/2$, for every $\varepsilon\leq\varepsilon_1$. In turn,
\[
\|x^{\star} \|^2 - \|x_{\varepsilon} \|^2  \leq 2\Vert x^{\star}\Vert \|x_{\varepsilon } -  x^{\star} \|, \quad \forall \varepsilon \geq 0 .
\]
Taking $\varepsilon^{\star}=\varepsilon_0$ and $C=2C_0\Vert x^{\star}\Vert$ and using \eqref{eq:ratetikhonov}, we get \eqref{eq:normratetikhonov}.
%, then we get
%\[\|x^{\star} \|^2 - \|x_{\varepsilon} \|^2  \leq C \varepsilon^{\frac{1}{2p}},\forall \varepsilon\in ]0,\varepsilon^{\star}].\]
%Example: Take $\varepsilon(t) = \frac{1}{t^r}$ with $r\leq 1$.

\end{proof}

\tcb{The previous proposition was the key to deriving a proper tuning of the parameter $\varepsilon(t)$ that satisfies all the conditions presented in Theorem~\ref{converge20}. In the following, we make this precise, recalling that the setting of Theorem~\ref{converge20} happened in the separable real Hilbert space $\H$, and that the previous proposition remains valid without separability.}
\begin{theorem}\label{practical}
Consider the setting of Theorem \ref{converge20} and suppose that  $F=f+g \in \EB^p(\calS_F)$. Then taking the Tikhonov parameter
 $\varepsilon(t) = \frac{1}{t^r}$ with
\[
1 \geq r > \frac{2p}{2p+1},
\]
then the three conditions \ref{t1}, \ref{t2}, and \ref{t3} of Theorem \ref{converge20}   are satisfied simultaneously.
In particular, the solution $X\in S_{\H}^{\nu}[t_0]$ of \eqref{CSGD} is unique and we get almost sure (strong) convergence of $X(t)$ to the minimal norm solution $x^{\star}=\proj_{\calS_F}(0)$.
\end{theorem}

\begin{proof}
It is direct to check \ref{t1} and \ref{t2}. In order to check \ref{t3}, let $\varepsilon^{\star}>0$ from Proposition \ref{ratetikhonov} and $T^{\star}=\max\pa{t_0,\left(\frac{1}{\varepsilon^{\star}}\right)^{\frac{1}{r}}}$, then we have 
\[
\|x^{\star} \|^2 - \|x_{\varepsilon(t)} \|^2  \leq C \frac{1}{t^\frac{r}{2p}}, \quad\forall t\geq T^{\star}.
\]
Therefore, 
\[
\int_{t_0}^{+\infty}\frac{\|x^{\star} \|^2 - \|x_{\varepsilon(t)} \|^2}{t^r}dt=\underbrace{\int_{t_0}^{T^{\star}}\frac{\|x^{\star} \|^2 - \|x_{\varepsilon(t)} \|^2}{t^r}dt}_{I_1}+\underbrace{\int_{T^{\star}}^{+\infty}\frac{\|x^{\star} \|^2 - \|x_{\varepsilon(t)} \|^2}{t^r}dt}_{I_2}.
\]
Is clear that $I_1$ is bounded (by $T^{\star}t_0^{-r}\Vert x^{\star}\Vert^2$ for instance). Hence \ref{t3} holds under the condition that
\[
\int_{T^{\star}}^{+\infty} \frac{1}{t^r}  \frac{C}{t^\frac{r}{2p}}dt  <+\infty ,
\]
which is true when $ r +  \frac{r}{2p} >1$, whence we deduce our condition $1 \geq r > \frac{2p}{2p+1}$.
\end{proof}

\subsection{Convergence rates of the objective in the smooth case}
%\tcb{In this section, we establish global convergence rates in expectation for the smooth convex case. Previously, we showed almost sure convergence in the Tikhonov setting, but no explicit convergence rate was provided. To derive such rates in the stochastic setting, it is useful to first recall a result from \cite{ACR}, where the authors analyze the deterministic case.}

\tcb{In this subsection, we turn our attention to the smooth case (\ie\ $g\equiv 0$) and derive explicit global convergence rates in expectation. Recall that in the previous subsections, we established almost sure strong convergence with Tikhonov regularization, but did not obtain any convergence rate. To fill this gap, at least in the smooth setting, we first revisit the result of \cite{ACR} which provides rates in the deterministic case. We then adapt their proof strategy--together with the stochastic analysis tools developed earlier-- to control the additional variance term and obtain the decay of the objective value and distance to the minimal norm solution in expectation. 
}

\smallskip

%Take $\varepsilon (t)=\displaystyle\frac{1}{t^{r} } $,  $0<r<1$, $t_0>0$. The convergence rate of the values and the strong convergence to the minimum norm solution is described below, see Attouch, Chbani, Riahi \cite[Theorem 5]{ACR}. 
%%
	\begin{theorem}\cite[Theorem 5]{ACR}\label{thm:model-a}
	Take $\varepsilon (t)=\displaystyle\frac{1}{t^{r} } $ and  $0<r< 1$. Let us consider \eqref{CSGDT1} in the case where $g\equiv 0$, \ie,
	\begin{equation}\label{eqr1}
	\dot{x}(t) + \nabla f\left(x(t) \right)+ \frac{1}{t^r} x(t)=0.
	\end{equation}	
 
	Let $x : [t_0, +\infty[ \to \H$ be a solution trajectory of \eqref{CSGDT1}. For $\varepsilon>0$ define $f_{\varepsilon}(x)\eqdef f(x)+\frac{\varepsilon}{2}\Vert x\Vert^2$, let $x_{\varepsilon}$ be the unique minimizer of $f_{\varepsilon}$, and consider the Lyapunov function \[E(t)\eqdef f_{\varepsilon(t)}(x(t))-f_{\varepsilon(t)}(x_{\varepsilon(t)})+\frac{\varepsilon(t)}{2}\Vert x(t)-x_{\varepsilon(t)}\Vert^2.\]
\smallskip
Then,  we have : 	
	\begin{renumerate}
	\item   $E(t) =  \mathcal O \left( \displaystyle\frac{1}{t}   \right) \mbox{ as } \; t \to +\infty;\hspace{8cm}\label{Lyap-basic2}$\\
	\item
	$f(x(t))-\min (f)= \mathcal O \left( \displaystyle\frac{1}{t^{r} }   \right) \mbox{ as } \; t \to +\infty;$\label{contr:fx(t)2bb}\\
	\item $\|x(t) -x_{\varepsilon(t)}\|^2=\mathcal{O}\left(\dfrac{1}{ t^{1-r}}\right) \mbox{ as } \; t \to +\infty.$ \label{contr:x(t)2b}
 \end{renumerate}
\end{theorem}
%In view of our Proposition \ref{ratetikhonov}, we can shed light on the convergence rate of the trajectory solution of \eqref{CSGDT1} towards the minimum norm solution. This result is new to the best of our knowledge. 

\tcb{In light of Proposition~\ref{ratetikhonov}, we can now characterize the rate at which the trajectory solution of \eqref{CSGDT1} converges to the minimum norm solution. To the best of our knowledge, this convergence rate estimate is new.}
\tcb{
\begin{corollary}\label{cor:model-a}
Consider the setting of Theorem \ref{thm:model-a}, then we have strong convergence of $x(t)$ to the minimum norm solution $x^{\star}=\proj_{\calS}(0)$. Moreover, if $f\in\EB^p(\calS)$, then
  \begin{equation}
      \|x(t) -x^{\star}\|^2=\begin{cases}
        \mathcal{O}\left(\dfrac{1}{ t^{\frac{r}{p}}}\right),&\quad \text{ if } r\in \left]0,\frac{p}{p+1}\right[;\\
        \mathcal{O}\left(\dfrac{1}{ t^{1-r}}\right),&\quad \text{ if } r\in \left[\frac{p}{p+1},1\right[
    \end{cases} \mbox{ as } \; t \to +\infty. \label{contr:x(t)2bbb}
    \end{equation}
\end{corollary}
\begin{proof}
    Combine the third item of Theorem~\ref{thm:model-a} and Proposition~\ref{ratetikhonov}.
\end{proof}
\begin{remark}
    We observe that the convergence rate is governed by a piecewise function, attaining its optimum when $r=\frac{p}{p+1},$ in which case we obtain \begin{equation*}
        \|x(t) -x^{\star}\|^2=\mathcal{O}\left(\frac{1}{t^{\frac{1}{p+1}}}\right) ,\mbox{ as } \; t \to +\infty.
    \end{equation*} 
    We also remark that this is strictly slower than the convergence rate of (deterministic) gradient flow when $f$ satisfies a H\"olderian error bound, in which case we have 
    \begin{equation*}
        \dist^2(x(t),\calS)=\mathcal{O}\left(\frac{1}{t^{\frac{2}{p}}}\right) ,\mbox{ as } \; t \to +\infty.
    \end{equation*} 
    This reflects the trade-off for ensuring strong convergence to the minimal norm solution with the Tikhonov regularization term $\varepsilon(t)=\frac{1}{t^{\frac{p}{p+1}}}, p\geq 1$.
\end{remark}
}
%We now have the necessary tools to present our main result in this context, which summarizes the global convergence rates in expectation satisfied by the trajectories of \eqref{CSGD} in the case where $g\equiv 0$. This result naturally generalizes to the stochastic setting the convergence rates given in Theorem \ref{thm:model-a} and Corollary \ref{cor:model-a}, since considering the noiseless case, \ie, $\sigma_{\infty}^2=0$, we recover the aforementioned results.

\tcb{We are ready now to state the main theorem of this subsection, which establishes global convergence rates in expectation for the trajectories of \eqref{CSGD} when $g\equiv0$. Moreover, this result recovers the deterministic convergence rates of Theorem~\ref{thm:model-a} and Corollary~\ref{cor:model-a}. Indeed, by setting $\sigma_{\infty}^2=0$, the stochastic term vanishes and we retrieve exactly the bounds of the deterministic case.
\begin{theorem}\label{importante1}
    Consider \eqref{CSGD} with $g\equiv 0$, $\varepsilon(t)=\frac{1}{t^r}$ where $0<r<1$, \ie 
    \begin{equation}\label{g0}\tag{$\mathrm{SDE-TA}$}
	\begin{cases}
	\begin{aligned}
	dX(t)&=-\nabla f(X(t))dt -\frac{1}{t^r}X(t)dt +\sigma(t,X(t))dW(t), \quad t\geq t_0;\\
	X(t_0)&=X_0.
	\end{aligned}
	\end{cases}
	\end{equation}
	where the initial data $X_0\in\Lp^{\nu}(\Omega;\H)$ for $\nu\geq 2$. 
	Assume that $f\in \Gamma_0(\H)\cap C_L^2(\H)$ with $\calS\eqdef\argmin(f) \neq \emptyset$, and $f\in \EB^p(\calS)$. Suppose also that $\sigma$ satisfies \eqref{H}, and that $\sigma_{\infty}\in\Lp^2([t_0,+\infty[)$ and is non-increasing. 
For $\varepsilon>0$, let $f_{\varepsilon}(x)\eqdef f(x)+\frac{\varepsilon}{2}\Vert x\Vert^2$, $x_{\varepsilon}$ be the unique minimizer of $f_{\varepsilon}$. Let $x^{\star}\eqdef \proj_{\calS}(0)$. Consider the energy function 
\[
E(t,x)\eqdef f_{\varepsilon(t)}(x)-f_{\varepsilon(t)}(x_{\varepsilon(t)})+\frac{\varepsilon(t)}{2}\Vert x-x_{\varepsilon(t)}\Vert^2,
\] 
and for $t_1>t_0$, 
\begin{equation}\label{eqdefr}
    R(t)\eqdef e^{-\frac{t^{1-r}}{1-r}}\int_{t_1}^t e^{\frac{s^{1-r}}{1-r}}\sigma_{\infty}^2(s)ds.
\end{equation}
Then, the solution trajectory $X\in S_{\H}^{\nu}[t_0]$ is unique, and the following holds:  	
\begin{renumerate}
    %\item $R(t)\rightarrow 0 \mbox{ as } \; t \to +\infty;$\hspace{8cm}\label{error}
    \item \label{crr} $R$ converges to $0$ at the rate, 
    \[
    R(t)=\mathcal{O}\left(\mathrm{exp}(-t^r(1-2^{-r}))+t^r\sigma_{\infty}^2\left(\frac{t_1+t}{2}\right)\right).
    \] 
    If, moreover, $\sigma_{\infty}^2(t)=\mathcal{O}(t^{-\alpha})$ for $\alpha>1$, then $R(t)=\mathcal{O}(t^{r-\alpha})$.
	\item \label{Lyap-basic3}
	$\EE[E(t,X(t))] =  \mathcal O \left( \displaystyle\frac{1}{t} + R(t)  \right)$ 
	\item\label{contr:fx(t)3bb}
	$\EE[f(X(t))-\min (f)]= \mathcal O \left( \displaystyle\frac{1}{t^{r} }+R(t) \right)$. 
	In addition, if $\sigma_{\infty}^2(t)=\mathcal{O}(t^{-\alpha})$ for $\alpha> 1$, then 
	\[
	\EE[f(X(t))-\min (f)]=
	\begin{cases}
	    \mathcal O \left( \displaystyle\frac{1}{t^{\alpha-r}}   \right),\quad& \text{if } \alpha\in ]1, 2r[;\\
        \mathcal O \left( \displaystyle\frac{1}{t^{r}}   \right),\quad& \text{if } \alpha\geq 2r;
	\end{cases} 
	\]
	\item \label{n5} 
	$\EE[\|X(t) -x_{\varepsilon(t)}\|^2]=\mathcal{O}\left(\dfrac{1}{ t^{1-r}}+t^{r}R(t)\right),$ which goes to $0$ as $t\rightarrow +\infty$ if $r\in ]0,\frac{1}{2}]$.
	\label{contr:x(t)3b} 
	If, moreover, $\sigma_{\infty}^2(t)=\mathcal{O}(t^{-\alpha})$ for $\alpha>\max\{2r,1\}$, then 
	\[
	\EE[\|X(t) -x_{\varepsilon(t)}\|^2]=\begin{cases}
	    \mathcal{O}\left(\dfrac{1}{ t^{\alpha-2r}}\right),&\quad \text{ if } \alpha\in ]\max\{2r,1\},r+1[;\\
     \mathcal{O}\left(\dfrac{1}{ t^{1-r}}\right),&\quad \text{ if } \alpha\geq r+1.
	\end{cases} .
	\]
    \item \label{n6} $\EE[\|X(t) -x^{\star}\|^2]=\mathcal{O}\left(\dfrac{1}{ t^{1-r}}+\dfrac{1}{ t^{\frac{r}{p}}}+t^{r}R(t)\right),$ which goes to $0$ as $t\rightarrow +\infty$ if $r\in ]0,\frac{1}{2}]$.  In addition, if $\sigma_{\infty}^2(t)=\mathcal{O}(t^{-\alpha})$ for $\alpha>\max\{2r,1\}$, then 
    \[
    \EE[\|X(t) -x^{\star}\|^2]=\mathcal{O}\left(\dfrac{1}{ t^{1-r}}+\dfrac{1}{ t^{\frac{r}{p}}}+\dfrac{1}{t^{\alpha-2r}}\right) .
    \]
    In particular,
    \[
    \EE[\|X(t) -x^{\star}\|^2]=\begin{cases}
        \mathcal{O}\left(\dfrac{1}{ t^{1-r}}\right),&\quad \text{ if } r\in \left]\frac{p}{p+1} , 1\right[,\alpha>r+1;\\
        \mathcal{O}\left(\dfrac{1}{ t^{\frac{r}{p}}}\right),&\quad \text{ if } r\in \left]0,\frac{p}{p+1}\right[,\alpha>\max\{1,\frac{r(2p+1)}{p}\};\\
        \mathcal{O}\left(\dfrac{1}{ t^{\alpha-2r}}\right),&\quad \text{ if } r\in \left]\frac{p}{2p+1} , 1\right[,\alpha\in \left(\max\{2r,1\},\min\{r+1,\frac{r(2p+1)}{p}\}\right).
    \end{cases}
    \]
	\end{renumerate}
\end{theorem}
}
%\begin{remark}\label{crbueno}
%    The expression in \ref{crr} goes to $0$ as $t\rightarrow +\infty$ since $\lim_{t\rightarrow\infty} t\sigma_{\infty}^2(t)=0$ and $r< 1$.
%\end{remark}
\begin{proof}
The existence and uniqueness of a solution was already stated in Theorem \ref{converge20}.\smallskip

The first item is a direct consequence of Lemma \ref{liiim0}, for the second one we recall that $\sigma_{\infty}\in\Lp^2([t_0,+\infty[)$ and is non-increasing, and we proceed as follows:
\begin{align*}
R(t)&=e^{-\frac{t^{1-r}}{1-r}}\int_{t_1}^{\frac{t_1+t}{2}} e^{\frac{s^{1-r}}{1-r}}\sigma_{\infty}^2(s)ds+e^{-\frac{t^{1-r}}{1-r}}\int_{\frac{t_1+t}{2}}^{t} e^{\frac{s^{1-r}}{1-r}}\sigma_{\infty}^2(s)ds \\
&\leq e^{\left(\frac{t_0}{2}\right)^r}e^{-t^r\left(1-2^{-r}\right)}\int_{t_1}^{+\infty}\sigma_{\infty}^2(s)ds+\sigma_{\infty}^2\left(\frac{t_1+t}{2}\right)D_{\frac{1}{1-r},1-r}\left(t\right),
\end{align*}
where \[D_{a,b}\left(t\right)=e^{-at^b}\int_0^{t}e^{as^b}ds.\] As a corollary of an upper bound of the Dawson integral shown in \cite[Section~7.8]{dawson}, we have that \[D_{a,b}(t)\leq \frac{2}{ab}t^{1-b},\quad 0<b\leq 2, a>0, t>0,\] thus we obtain 
\[
R(t)=\mathcal{O}\left(\mathrm{exp}(-t^r(1-2^{-r}))+t^r\sigma_{\infty}^2\left(\frac{t_1+t}{2}\right)\right) .
\]
%The the right hand side goes to $0$ goes to $0$ as $t\rightarrow +\infty$ since $\lim_{t\rightarrow\infty} t\sigma_{\infty}^2(t)=0$ and $r< 1$. 
Since $\sigma_{\infty}^2$ is non-increasing and $r< 1$, we have  
\[
0\leq t\sigma_{\infty}^2(t)\leq 2\int_{\frac{t}{2}}^t\sigma_{\infty}^2(u)du,
\]
and the right hand side goes to $0$ as $t\rightarrow +\infty$ since $\sigma_{\infty}\in\Lp^2([t_0,+\infty[)$ by assumption. Thus we obtain that $\lim_{t\rightarrow\infty}t\sigma_{\infty}^2(t)=0$ which proves claim \ref{crr}.

\smallskip

The remainder of the proof follows by applying It\^o's formula to \eqref{g0} with the function
\[
\phi(t,x) \eqdef  \Phi(t)E(t,x) \quad \text{where} \quad \Phi(t)\eqdef \exp\Bigl(\!\int_{t_1}^t s^{-r} ds\Bigr),
\]
and then taking expectation. Following similar computations as in \cite[Theorem~3]{ACR}, we obtain
\begin{align*}
    \mathbb{E}[\phi(t,X(t))]&\leq\mathbb{E}[\phi(t_0,X_0)]-\Vert x^{\star}\Vert^2\int_{t_0}^t \dot{\varepsilon}(s)\Phi(s)ds+\int_{t_0}^t\mathbb{E}\left[\tr[\sigma(s,X(s))\sigma^\star(s,X(s))\nabla^2 \phi(s,X(s))]\right]ds\\
    &\leq \mathbb{E}[\phi(t_0,X_0)]-\Vert x^{\star}\Vert^2\int_{t_0}^t \dot{\varepsilon}(s)\Phi(s)ds+\left(L+2t_0^{-r}\right)\int_{t_0}^t\Phi(s)\sigma_{\infty}^2(s)ds.
\end{align*}
Dividing by $\Phi(t)$, we have equivalently that 
\begin{align*}
    \mathbb{E}[E(t,X(t))]&\leq\frac{\mathbb{E}[\phi(t_0,X_0)]}{\Phi(t)}-\frac{\Vert x^{\star}\Vert^2}{\Phi(t)}\int_{t_0}^t \dot{\varepsilon}(s)\Phi(s)ds+\frac{(L+2t_0^{-r})}{\Phi(t)}\int_{t_0}^t\Phi(s)\sigma_{\infty}^2(s)ds\\
    &\leq \frac{\mathbb{E}[\phi(t_0,X_0)]}{\Phi(t)}+\frac{\Vert x^{\star}\Vert^2}{\rho t}+\frac{(L+2t_0^{-r})}{\Phi(t)}\int_{t_0}^t\Phi(s)\sigma_{\infty}^2(s)ds,
\end{align*}
for $\rho<\frac{1}{r}$ (see the proof of \cite[Theorem 5]{ACR}). By definition $\frac{1}{\Phi(t)}\int_{t_0}^t\Phi(s)\sigma_{\infty}^2(s)ds=R(t)$, and since $\frac{1}{\Phi(t)}$ decays exponentially, we conclude with the claim of item \ref{Lyap-basic3}.
%\begin{equation*}
%    \mathbb{E}[E(t,X(t))]=\mathcal{O}\left(\frac{1}{t}+R(t)\right).
%\end{equation*}
Besides, by \cite[Lemma~3]{ACR}, we have 
\begin{equation*}
    f(x)-\min f \leq E(t,x)+\frac{\varepsilon(t)}{2}\Vert x^{\star}\Vert^2 \quad \text{and} \quad
    \Vert x-x_{\varepsilon(t)}\Vert^2 \leq \frac{E(t,x)}{\varepsilon(t)}.
\end{equation*}
Taking expectation and inserting the bound of \ref{Lyap-basic3}, we obtain claims \ref{contr:fx(t)3bb} and \ref{n5}. 
%\begin{align*}
%    \mathbb{E}[f(X(t))-\min f]&=\mathcal{O}\left(\frac{1}{t^r}+R(t)\right),\\
%    \mathbb{E}[\Vert X(t)-x_{\varepsilon(t)}\Vert^2]&=\mathcal{O}\left(\frac{1}{t^{1-r}}+t^{r}R(t)\right).
%\end{align*}
Finally, for item \ref{n6}, we combine \ref{n5} and Proposition~\ref{ratetikhonov}.
%to obtain that 
%\begin{equation*}
%    \mathbb{E}[\Vert X(t)-x^\star\Vert^2]=\mathcal{O}\left(\frac{1}{t^{1-r}}+t^{r}R(t)+\frac{1}{t^{\frac{r}{p}}}\right).
%\end{equation*} 
To conclude with the particular convergence rates, we plug in the derived rates for $R(t)$ obtained in \ref{crr} and the rate of $\sigma_{\infty}^2(t)$.

\end{proof}

\begin{remark}
In the finite-dimensional case, \ie, $\H=\R^d$ (not necessarily $\K$), we can weaken the assumption $f\in C_L^2(\H)$ to $f\in C_L^{1,1}(\H)$ thanks to \cite[Proposition 2.2]{mio}.
\end{remark}

\begin{remark}
Comparing Theorem~\ref{importante1} to its deterministic counterpart Theorem~\ref{thm:model-a} (see also Corollary~\ref{cor:model-a}), one has the additional term $R(t)$ that appears in each rate. This necessarily slows down the convergence rate compared to the deterministic setting. But as expected, it is the price to be paid to account for stochastic noise while ensuring convergence.
\end{remark}

\begin{remark}
Our result in Theorem~\ref{converge20} ensures that with the Tikhonov regularization, the solution trajectory strongly converges in almost sure sense to the minimal norm solution, provided that the regularization coefficient $\varepsilon(t)$ is well chosen (verifies \ref{t1}, \ref{t2} and \ref{t3}), and the diffusion term decays fast enough. While it is easy to choose $\varepsilon(t)$ so that \ref{t1}-\ref{t2} hold, fulfilling \ref{t3} required more involved arguments, and for instance that $f$ verifies a H\"olderian error bound. This also allowed to derive the (pointwise) convergence rates in expectation of Theorem~\ref{importante1}. These quantitative estimates reveal that there is a trade-off between the decay of $\varepsilon(t)$ and that of the diffusion term $\sigma_{\infty}(t)$ in order to maintain convergence and have meaningful convergence rates. This is clearly reflected in the form of the function $R(t)$ in Theorem~\ref{importante1}\ref{crr}. For instance, mere square-integrability of $\sigma_{\infty}(t)$ is not sufficient as $\sigma_{\infty}(t)$ must decrease at least as $t^{-\alpha}$, $\alpha > 1$.
\end{remark}

\section{Conclusion, Perspectives} 

The purpose of this work was to study the convergence properties of trajectories of subgradient-like flows under stochastic errors in infinite dimensional separable real Hilbert spaces. The motivation stems from solving non-smooth convex optimization problems with noisy subgradient oracles with vanishing variance. We have shown important properties of these trajectories, such as the almost sure weak (resp. strong) convergence to a minimizer (resp. minimal norm one) without (resp. with) Tikhonov regularization. We have also investigated the convergence rates and highlighted the trade-off between the tuning of the Tikhonov regularization coefficient and the noise variance. This work leads us to important extensions, among which,we mention the following ones:
\begin{itemize}
    \item Extend our results, with and without Tikhonov regularization, to the case of to the case of operators where $\nabla f$ and $\partial g$ are replaced by, respectively, a co-coercive operator $B$ and a maximal monotone operator $A$. 
    %and $\partial g$ by $B$ Theorem~\ref{converge} to the case of maximal monotone operators, \ie, replace $\nabla f$ by $A$ (a maximal monotone operator) and $\partial g$ by $B$ (a set-valued maximal monotone operator) in the dynamic and prove the almost sure weak convergence to a zero of $A+B$.
    \item Investigate the transition to second-order dynamics via time-scaling and averaging, and analyzing its corresponding convergence properties.
    \item Study second-order dynamics with inertia in view of understanding the behavior of accelerated dynamics in the presence of stochastic errors. %This investigation would involve an independent Lyapunov analysis. 
\end{itemize}
Some of these aspects are already the subject of ongoing research work.

%%%%%%%%%%%%APPENDIX%%%%%%%%%%%%

\appendix

\section{Auxiliary results}\label{aux}
\subsection{Deterministic results}
\begin{lemma}\label{lim0}
Let $t_0\geq 0$ and $a,b:[t_0,+\infty[\rightarrow \R_+$. If $\lim_{t\rightarrow \infty} a(t)$ exists, $b\notin\Lp^1([t_0,+\infty[)$ and $\int_{t_0}^\infty a(s)b(s)ds<+\infty$, then $\lim_{t\rightarrow \infty} a(t)=0.$
\end{lemma}

\begin{lemma}[Comparison Lemma]\label{comparison}
Let $t_0\geq 0$ and $T> t_0$. Assume that $h:[t_0,+\infty[\rightarrow\R_+$ is measurable with $h \in \Lp^1([t_0,T])$ , that $\psi:\R_+\rightarrow\R_+$ is continuous and non-decreasing, $\varphi_0>0$ and the Cauchy problem
\begin{equation*}
\begin{cases}
\varphi'(t)=-\psi(\varphi(t)) + h(t) & \text{for almost all $t\in [t_0,T]$}\\
\varphi(t_0)=\varphi_0
\end{cases}
\end{equation*}
has an absolutely continuous solution $\varphi:[t_0,T]\rightarrow \R_+$. If a bounded from below lower semicontinuous function $\omega:[t_0,T]\rightarrow \R_+$ satisfies 
\[
\omega(t)\leq\omega(s)-\int_s^t \psi(\omega(\tau))d\tau + \int_s^t h(\tau)d\tau
\]
for $t_0\leq s < t \leq T$ and $\omega(t_0)=\varphi_0$, then 
\[
\omega(t)\leq \varphi(t)\quad \text{for $t\in [t_0,T]$}.
\]
\end{lemma}

%\begin{theorem}[Egorov's Theorem] \cite[Chapter 3, Exercise 16]{rudin} \label{egorov}
% If $\mu(X)<\infty$ and $(f_t)_{t\in \R_+}$ is a family of real functions such that for all $x\in X$: 
% \begin{enumerate}
%     \item $\lim_{t\rightarrow\infty} f_t(x)=f(x)$ and
%     \item $t\rightarrow f_t(x)$ is continuous.
% \end{enumerate} 
% Then, for every $\delta>0$, there exists a measurable set $E_{\delta}\subset X$, with $\mu(X\setminus E_{\delta})<\delta$, such that $(f_t)_{t\in\R_+}$ converges uniformly on $E_{\delta}$. 
%\end{theorem}

\begin{lemma} \label{existenceof} 
Let $f:\R_+\rightarrow\R$ and $\liminf_{t\rightarrow\infty} f(t)\neq \limsup_{t\rightarrow\infty} f(t)$. Then there exists a constant $\alpha$, satisfying $\liminf_{t\rightarrow\infty} f(t)< \alpha<\limsup_{t\rightarrow\infty} f(t)$, such that for every $\beta>0$, we can define a sequence $(t_k)_{k\in\N}\subset\R$ such that \[f(t_k)>\alpha,\hspace{0.3cm} t_{k+1}>t_k+\beta\hspace{0.3cm} \forall k\in\N.\]
\end{lemma}
\begin{proof}
    See proof in \cite[Lemma A.3]{mio}.
\end{proof}

\begin{lemma}\label{liiim0}
    Take $t_0>0$, and let $f\in \Lp^1([t_0,+\infty[)$ be continuous. Consider a non-decreasing function $\varphi:[t_0,+\infty[\rightarrow\R_+$ such that $\lim_{t\rightarrow+\infty}\varphi(t)=+\infty$. Then $\lim_{t\rightarrow+\infty}\frac{1}{\varphi(t)}\int_{t_0}^t\varphi(s)f(s)ds=0$.
\end{lemma}
\begin{proof}
    See proof in \cite[Lemma A.5]{hessianpert}
\end{proof}
%\begin{proof}
% Since $\liminf_{t\rightarrow\infty} f(t)$ and $\limsup_{t\rightarrow\infty} f(t)$ are different real numbers, there exists $\alpha$ such that $$\liminf_{t\rightarrow\infty} f(t)<\alpha<\limsup_{t\rightarrow\infty} f(t),$$ besides by definition of $\limsup$, there exists a sequence $(t_k)_{k\in\N}$ such that $\lim_{k\rightarrow\infty} t_k=\infty$ and $f(t_k)>\alpha$. Let $\beta>0$ and $n_0=0$, let us define recursively for $j\geq 1$, $n_j=\min\{n>n_{j-1}: t_n-t_{n_{j-1}}>\beta\}$. Let $j'\in\N$ the first natural such that $n_{j'}=\infty$, this implies that for every $n>n_{j'-1}$, $t_n\leq\beta+t_{n_{j'-1}}<\infty$, a contradiction since $\lim_{n\rightarrow\infty}t_n=\infty$, then for every $j\in\N$, $n_j<\infty$. Thus, we define $(t_{n_j})_{j\in\N}$ a subsequence of $(t_k)_{k\in\N}$ such that $\lim_{j\rightarrow\infty}t_{n_j}=\infty$ and for every $j\in\N$, $t_{n_{j+1}}-t_{n_j}>\beta$.
%\end{proof}

\subsection{Stochastic results}
\subsubsection{On stochastic processes}\label{onstochastic}
Let us recall some elements of stochastic analysis. Throughout the paper, $(\Omega,\mathcal F,\mathbb P)$ is a probability space and $\{{\mathcal F}_t| t\geq 0\}$ is a filtration of the $\sigma-$algebra $\mathcal F$. Given $\mathcal{C}\in\mathcal P(\Omega)$, we will denote $\sigma(\mathcal{C})$ the $\sigma-$algebra generated by $\mathcal{C}$. We denote $\mathcal F_{\infty}\eqdef \sigma \left(\bigcup_{t\geq 0} \mathcal F_t \right)\in\mathcal F$.

The expectation of a random variable $\xi:\Omega\rightarrow\H$ is denoted by 
\[
\EE(\xi)\eqdef \int_{\Omega}\xi(\omega)d\PP(\omega).
\]
An event $E\in\calF$ happens almost surely if $\PP(E)=1$, and it will be denoted as "$E$, $\PP$-a.s." or simply "$E$, a.s.". The characteristic function of an event $E\in\calF$ is denoted by 
\[
\ind_E (\omega) \eqdef
\begin{cases}
1 & \text{if } \omega\in E,\\
0 & \text{otherwise}.
\end{cases}
\] 
%And $\EE(\ind_E)=\Pro(E)$. \smallskip
An $\H$-valued stochastic process starting at $t_0\geq 0$ is a function $X:\Omega\times[t_0,+\infty[\rightarrow\H$. It is said to be continuous if $X(\omega,\cdot)\in C([t_0,+\infty[;\H)$ for almost all $\omega\in\Omega$. We will denote $X(t)\eqdef X(\cdot,t)$. We are going to study SDEs, and in order to ensure the uniqueness of a solution, we introduce an equivalence relationship over stochastic processes. Two stochastic processes $X,Y:\Omega\times [t_0,T]\rightarrow\H$ are said to be equivalent if $X(t)=Y(t)$, $\forall t\in [t_0,T]$, $\PP$-a.s. This leads us to define the equivalence relation $\calR$, which associates the equivalent stochastic processes in the same class. 

\smallskip

Furthermore, we will need some properties about the measurability of these processes. A stochastic process $X:\Omega\times [t_0,+\infty[\rightarrow\H$ is progressively measurable if for every $t\geq t_0$, the map $\Omega\times[t_0,t]\rightarrow\H$ defined by $(\omega,s)\rightarrow X(\omega,s)$ is $\calF_t\otimes\calB([t_0,t])$-measurable, where $\otimes$ is the product $\sigma$-algebra and $\calB$ is the Borel $\sigma$-algebra. On the other hand, $X$ is $\calF_t$-adapted if $X(t)$ is $\calF_t$-measurable for every $t\geq t_0$. It is a direct consequence of the definition that if $X$ is progressively measurable, then $X$ is $\calF_t$-adapted.

\smallskip

%With these concepts, we can introduce some interesting spaces.
Let us define the quotient space:
\[
S_{\H}^0[t_0,T] \eqdef \left\lbrace X: \Omega\times[t_0,T]\rightarrow\H, \; X \text{ is a prog. measurable cont. stochastic process}\right\rbrace\Big/\calR.
\]
Set $S_{\H}^0[t_0]\eqdef \bigcap_{T\geq t_0} S_{\H}^0[t_0,T]$.
For $\nu>0$, we define $S_{\H}^{\nu}[t_0,T]$ as the subset of processes $X(t)$ in $S_{\H}^0[t_0,T]$ such that 
\[
S_{\H}^{\nu}[t_0,T] \eqdef  \left\lbrace X\in S_{\H}^0[t_0,T]:  \;
 \EE\pa{\sup_{t\in[t_0,T]}\norm{X_t}^{\nu}}<+\infty \right\rbrace.
\] 
We define $S_{\H}^{\nu}[t_0] \eqdef \bigcap_{T\geq t_0} S_{\H}^{\nu}[t_0,T]$. %and $S_{\H}^{\nu}\eqdef S_{\H}^{\nu}[0]$.
\smallskip

%Let $I\subseteq \N$ be a numerable set such that $\{e_i\}_{i\in I}$ is an orthonormal basis of $\K$, and $\{w_i(t)\}_{i\in I, t\geq 0}$ be a sequence of independent Brownian motions defined on the filtered space $(\Omega,\calF,\calF_t,\Pro)$. The process \[W(t)=\sum_{i\in I} w_i(t)e_i\]
%is well-defined (independent from the election of $\{e_i\}_{i\in I}$) and is called a $\K$-valued cylindrical Brownian motion. Besides, let $G:\Omega\times\R_+\rightarrow\calL_2(\K;\H)$ be a measurable and $\calF_t$-adapted process, then we can define $\int_0^t G(s)dW(s)$ which is the stochastic integral of $G$, and we have that $G\rightarrow\int_0^{\cdot}G(s)dW(s)$ is an isometry between the measurable and $\calF_t-$adapted $\calL_2(\K;\H)-$valued processes and the space of $\H$-valued continuous square-integrable martingales (see \cite[Theorem 2.3]{infinite}).

Following the notation of \cite[Section 2.1.2]{infinite}, we say that $W_t$ is a $\K$-valued cylindrical Brownian motion defined on on the filtered space $(\Omega,\calF,\calF_t,\Pro)$ if: \begin{renumerate}
    \item For an arbitrary $t\geq 0$, the mapping $W_t:\K\rightarrow \Lp^2(\Omega;\R)$ is linear;
    \item For an arbitrary $k\in\K$, $W_t(k)$ is an $\mathcal{F}_t$ Brownian motion;
    \item For arbitrary $k,k'\in\K$ and $t\geq 0$, $\mathbb{E}[W_t(k)W_t(k')]=t\langle k,k'\rangle_{\K}$.
\end{renumerate}  
\begin{remark}
    There is no $\K$-valued process $\tilde{W}_t$ such that:
    $$W_t(k)(\omega)=\langle \tilde{W}_t(\omega),k\rangle_{\K}.$$
    However, if $Q$ is a non-negative definite symmetric trace-class operator on $\K$, then a $\K-$valued $Q-$Brownian motion can be defined (see e.g. \cite[Definition 2.6]{infinite}, \cite[Section 4.1]{daprato}). Moreover, if $\K = \R^m$, then $W_t(k) = \langle \tilde{W}_t, k \rangle_{\K}$, where $\tilde{W}_t$ denotes the standard $m$-dimensional Brownian motion. Thus, the $\R^m$-cylindrical Brownian motion coincides with the standard $m$-dimensional Brownian motion.
\end{remark}

Besides, let $G:\Omega\times\R_+\rightarrow\calL_2(\K;\H)$ be a measurable and $\calF_t$-adapted process, then we can define $\int_0^t G(s)dW(s)$ which is the stochastic integral of $G$, and we have that $G\rightarrow\int_0^{\cdot}G(s)dW(s)$ is an isometry between the measurable and $\calF_t-$adapted $\calL_2(\K;\H)-$valued processes and the space of $\H$-valued continuous square-integrable martingales (see \cite[Theorem 2.4]{infinite}).

\tcb{
\subsection{Proof of Theorem~\ref{cps}}\label{sec:proofcps}
\begin{proof}
Set $F_{\varepsilon} (x) \eqdef F(x) + \frac{\varepsilon}{2}\|x\|^2$. Then \eqref{CSGDT1} can be written equivalently in a a compact form as
\[
\dot{x}(t) + \partial F_{\varepsilon (t)} (x(t)) \ni 0.
\]
Set $h (t) \eqdef \frac{1}{2}\|x(t) -x^{\star} \|^2$ where $x^{\star} = \proj_{\calS_F} (0)$. Derivation of $f$ and  the constitutive equation \eqref{CSGDT1} give
\begin{equation}\label{dert-1}
\dot{h}(t) + \left\langle  -\dot{x}(t), x(t)- x^{\star} \right\rangle=0,
\end{equation}
where $-\dot{x}(t)\in \partial F_{\varepsilon (t)} (x(t))$. By strong convexity of $F_{\varepsilon (t)}$, we get
\[
F_{\varepsilon (t)}(x^{\star}) \geq F_{\varepsilon (t)}(x(t)) + \left\langle  y(t),  x^{\star} -x(t)\right\rangle +  \frac{\varepsilon (t)}{2} \|x(t) -x^{\star} \|^2,
\]
for every $y(t)\in \partial F_{\varepsilon (t)} (x(t))$. Using that $F_{\varepsilon (t)}(x(t))\geq F_{\varepsilon (t)}(x_{\varepsilon (t)})$, we get
\[
F(x^{\star}) + \frac{\varepsilon(t)}{2} \|x^{\star} \|^2     \geq F(x_{\varepsilon (t)}) + 
 \frac{\varepsilon(t)}{2} \|x_{\varepsilon (t)} \|^2 +\left\langle  y(t),  x^{\star} -x(t)\right\rangle +  \frac{\varepsilon (t)}{2} \|x(t) -x^{\star} \|^2,\]
 for every $y(t)\in \partial F_{\varepsilon (t)} (x(t))$. Besides, from $F(x^{\star}) \leq F(x_{\varepsilon (t)})$ we deduce
\begin{equation}\label{dert-2}
\left\langle  y(t),  x(t)- x^{\star} \right\rangle \geq   \varepsilon (t) h(t) + \frac{\varepsilon (t)}{2} \left(    \|x_{\varepsilon (t)} \|^2 -  \|x^{\star} \|^2 \right),
\end{equation}
for every $y(t)\in \partial F_{\varepsilon (t)} (x(t))$. Combining \eqref{dert-1} with \eqref{dert-2} we obtain
\begin{equation}\label{xt3}
\dot{h}(t) + \varepsilon (t) h (t) \leq \frac{1}{2}\varepsilon (t) \left( \|x^{\star} \|^2  - \| x_{\varepsilon (t)} \|^2 \right).
\end{equation}
Set $m(t) \eqdef \exp{\int_{t_0}^t \varepsilon (s)ds}$. Integrating \eqref{xt3} from $t_0$ to $t$, we get
\begin{equation}\label{eq: Tikh1}
h(t)  \leq \frac{h (t_0) }{m(t)} +  \frac{1}{2 m(t)} \int_{t_0}^t  m'(s) \left( \|x^{\star} \|^2  - \| x_{\varepsilon (s)} \|^2 \right)ds.
\end{equation}
According to  hypothesis \ref{tikzero} and the classical property of the Tikhonov approximation we have \linebreak
$ x_{\varepsilon (t)} \to x^{\star}$, and hence $ \|x^{\star} \|^2  - \| x_{\varepsilon (s)} \|^2  \to 0$.
To pass to the limit on \eqref{eq: Tikh1} we use hypothesis \ref{tikinf} which tells us that $m(t) \to +\infty$. 
Let us complete the argument by using that convergence implies ergodic convergence. Precisely, given $\delta >0$, let $t_{\delta} >t_0$ such that $| \|x^{\star} \|^2  - \| x_{\varepsilon (s)} \|^2 | \leq \delta$ 
for $s \geq t_{\delta}$.  Then split the integral as follows
\begin{eqnarray}
h(t)  &\leq & \frac{h (t_0) }{m(t)} +  \frac{1}{2 m(t)} \int_{t_0}^{t_{\delta}}  m'(s) \left( \|x^{\star} \|^2  - \| x_{\varepsilon (s)} \|^2 \right)ds +  \delta \frac{1}{2 m(t)} \int_{t_{\delta}}^{t}  m'(s) ds \\
&\leq &  \frac{h (t_0) }{m(t)} +  \frac{1}{2 m(t)} \int_{t_0}^{t_{\delta}}  m'(s) \left( \|x^{\star} \|^2  - \| x_{\varepsilon (s)} \|^2 \right)ds + \frac{\delta}{2}.
\end{eqnarray}
Then let $t$ tend to infinity, to get $\limsup_{t \to +\infty} h(t) \leq \frac{\delta}{2}$.
This being true for any $\delta >0$ gives the result.
\end{proof}
}

\subsection{\tcb{Existence and uniqueness of the SDI}}\label{exiuniappendix}
\tcb{
We now prove Theorem~\ref{exiuni}, which specifies conditions ensuring the existence and uniqueness of a solution to \eqref{SDI}. The argument builds on prior results while addressing aspects not covered in \cite{petterson}.
\begin{proof}
The existence of a solution $(X,\eta)$ in the sense of Definition~\ref{def:SDIsolution} comes from \cite[Theorem~3.5]{petterson} (see \cite[Section 7.1.1]{daprato} for the SDE case). We now turn to uniqueness. Let $(X_1,\eta_1)$ and $(X_2,\eta_2)$ be two solutions of \eqref{SDI}. By It\^o's formula, we have 
\begin{align*}
\Vert X_1(t)-X_2(t)\Vert^2&=2\int_{t_0}^t \langle b(s,X_1(s))-b(s,X_2(s)), X_1(s)-X_2(s)\rangle ds\\
&-2\int_{t_0}^t \langle \eta_1'(s)-\eta_2'(s), X_1(s)-X_2(s) \rangle ds
    +\int_{t_0}^t \Vert \sigma(s,X_1(s))-\sigma(s,X_2(s))\Vert_{\mathrm{HS}}^2 ds\\
    &+\int_{t_0}^t \langle X_1(s)-X_2(s),[\sigma(s,X_1(s))-\sigma(s,X_2(s))] dW(s)\rangle.
\end{align*}
Since for almost all $t \geq 0$, $\eta_i'(t)\in A(X_i(t))$, $i=\{1,2\}$, by monotonicity of $A$, we have that for almost all $t\geq t_0$, 
\[
\langle \eta_1'(t)-\eta_2'(t), X_1(t)-X_2(t) \rangle \geq 0,
\] 
and thus the second term on the right-hand side is non-positive.
Now, let $n\in \N$ arbitrary and consider the stopping time $\tau_n=\inf\{t\geq t_0: \Vert X_1(t)-X_2(t)\Vert\geq n\}$ and evaluate the previous equation at $t\wedge \tau_n$, denoting $X_i^n(t)=X_i(t\wedge \tau_n)$ ($i=\{1,2\}$), we have 
\begin{align*}
\Vert X_1^n(t)-X_2^n(t)\Vert^2
&\leq  2\int_{t_0}^{t\wedge\tau_n} \langle b(s,X_1(s))-b(s,X_2(s)), X_1(s)-X_2(s)\rangle ds \\
&+\int_{t_0}^{t\wedge\tau_n} \Vert \sigma(s,X_1(s))-\sigma(s,X_2(s))\Vert_{\mathrm{HS}}^2 ds \\
&+\int_{t_0}^{t\wedge\tau_n} \langle X_1(s)-X_2(s),[\sigma(s,X_1(s))-\sigma(s,X_2(s))] dW(s)\rangle\\
& \leq L(L+2)\int_{t_0}^{t\wedge\tau_n} \Vert X_1(s)-X_2(s)\Vert^2 ds\\
&+\int_{t_0}^{t\wedge\tau_n} \langle X_1(s)-X_2(s),[\sigma(s,X_1(s))-\sigma(s,X_2(s))] dW(s)\rangle\\
& \leq L(L+2)\int_{t_0}^{t} \Vert X_1^n(s)-X_2^n(s)\Vert^2 ds\\
&+\int_{t_0}^{t\wedge\tau_n} \langle X_1(s)-X_2(s),[\sigma(s,X_1(s))-\sigma(s,X_2(s))] dW(s)\rangle.
\end{align*}
Note that we have used Cauchy-Schwarz inequality and the Lipschitz assumption on $(b,\sigma)$ in the second inequality. Taking expectation of both sides and using the properties of It\^o's integral we obtain 
\begin{align*}
    \EE(\Vert X_1^n(t)-X_2^n(t)\Vert^2)&\leq L(L+2) \int_{t_0}^{t} \EE(\Vert X_1^n(s)-X_2^n(s)\Vert^2) ds.
\end{align*}
By Gr\"onwall's inequality, we obtain that 
\[
\EE(\Vert X_1^n(t)-X_2^n(t)\Vert^2)= 0, \forall t\geq t_0,\forall n\in\N.
\] 
On the other hand, we have that $\lim_{n\rightarrow+\infty} t\wedge\tau_n=t$. Therefore,  taking $\liminf_{n\rightarrow+\infty}$ in the previous expression, using Fatou's Lemma and the fact that $X_1,X_2$ are a.s. continuous processes, we conclude that $\EE(\Vert X_1(t)-X_2(t)\Vert^2)= 0$, consequently 
\[
\Pro(X_1(t)=X_2(t), \forall t\in [t_0,T])=1,\quad \text{ for every } T>t_0.
\]
Let $T>t_0$ arbitrary, let us prove that $\EE\left(\sup_{t\in[t_0,T]}\Vert X(t)\Vert^{\nu}\right)<+\infty$. Using It\^o's formula with the solution process $X$ and the anchor function $\phi(x)=\Vert x-x^{\star}\Vert^2$ for $x^{\star}\in A^{-1}(0)$, we obtain for every $t\in [t_0,T]$:
\begin{align*}
\Vert X(t)-x^{\star}\Vert^2 & = \Vert X_0-x^{\star}\Vert^2+2\int_{t_0}^t \langle b(s,X(s)),X(s)-x^{\star}\rangle ds-2\int_{t_0}^t \langle \eta'(s),X(s)-x^{\star}\rangle ds\\
&+\int_{t_0}^t \Vert \sigma(s,X(s))\Vert_{\mathrm{HS}}^2 ds+2\int_{t_0}^t \langle X(s)-x^{\star},\sigma(s,X(s))dW(s)\rangle.
\end{align*}
Since $\eta'(t)\in A(X(t))$ for almost all $t \geq 0$, and $0\in A(x^{\star})$, by monotonicity of $A$ we have that for every $t\in [t_0,T]$, 
\[
\langle \eta'(t),X(t)-x^{\star}\rangle\geq 0,\quad \text{for almost all } t \geq 0 .
\]  
Thus the second integral is nonnegative, which implies 
\begin{equation}\label{sdv}
\begin{aligned}
\Vert X(t)-x^{\star}\Vert^2 &\leq \Vert X_0-x^{\star}\Vert^2+2\int_{t_0}^t \langle b(s,X(s)),X(s)-x^{\star}\rangle ds+\int_{t_0}^t \Vert \sigma(s,X(s))\Vert_{\mathrm{HS}}^2 ds\\
&+2\int_{t_0}^t \langle X(s)-x^{\star},\sigma(s,X(s))dW(s)\rangle.
\end{aligned}
\end{equation}
Moreover, we have
\[
2\langle b(t,x),x-x^{\star}\rangle+\Vert\sigma(t,x)\Vert_{\mathrm{HS}}^2\leq 2\Vert b(t,x)\Vert \Vert x-x^{\star}\Vert+\Vert\sigma(t,x)\Vert_{\mathrm{HS}}^2\leq C(1+\Vert x-x^{\star}\Vert^2),\quad\forall t\geq t_0,\forall x\in\H.
\] 
We now proceed as in the proof of \cite[Lemma 3.2]{gig} to conclude that $X\in S_{\H}^{\nu}[t_0]$. In fact, we take power $\frac{\nu}{2}$ at both sides of \eqref{sdv}, then using that $(a+b+c)^{\frac{\nu}{2}}\leq 3^{{\frac{\nu-2}{2}}}(a^{\frac{\nu}{2}}+b^{\frac{\nu}{2}}+c^{\frac{\nu}{2}})$ we have 
\begin{align*}
    \Vert X(t)-x^{\star}\Vert^{\nu} &\leq 3^{{\frac{\nu-2}{2}}}\left(\Vert X_0-x^{\star}\Vert^{\nu}+C^{\frac{\nu}{2}}\left(\int_{t_0}^t 1+\Vert X(s)-x^{\star}\Vert^2 ds\right)^{\frac{\nu}{2}}\right)\\
    &+3^{{\frac{\nu-2}{2}}}2^{\frac{\nu}{2}}\left(\int_{t_0}^t \langle X(s)-x^{\star},\sigma(s,X(s))dW(s)\rangle\right)^{\frac{\nu}{2}}.
\end{align*}
Now taking supremum $t\in [t_0,T]$ and then expectation at both sides, we have that there exists $K=K(\nu,T)$ such that:
\begin{align*}
    \EE\left(\sup_{t\in [t_0,T]}\Vert X(t)-x^{\star}\Vert^{\nu}\right) &\leq K\left(1+\EE\left(\Vert X_0-x^{\star}\Vert^{\nu}\right)+\int_{t_0}^T \EE(\Vert X(s)-x^{\star}\Vert^{\nu}) ds\right)\\
    &+K\EE\left(\sup_{t\in [t_0,T]}\Big|\int_{t_0}^t \langle X(s)-x^{\star},\sigma(s,X(s))dW(s)\rangle\Big|^{\frac{\nu}{2}}\right).
\end{align*}
By Proposition \ref{burkholder1}, we get that, for a redefined $K=K(\nu,T)$,
\begin{equation}\label{sdn1}
    \begin{aligned}
    \EE\left(\sup_{t\in [t_0,T]}\Vert X(t)-x^{\star}\Vert^{\nu}\right) &\leq K\left(1+\EE\left(\Vert X_0-x^{\star}\Vert^{\nu}\right)+\int_{t_0}^T \EE(\Vert X(s)-x^{\star}\Vert^{\nu}) ds\right)\\
    &+K\EE\left(\Big|\int_{t_0}^T \Vert X(s)-x^{\star}\Vert^2\Vert\sigma(s,X(s))\Vert_{\mathrm{HS}}^2 ds\Big|^{\frac{\nu}{4}}\right).    
    \end{aligned}
\end{equation}
Note that by Cauchy-Schwarz and Young's inequality,
\begin{align*}
    &\EE\left(\Big|\int_{t_0}^T \Vert X(s)-x^{\star}\Vert^2\Vert\sigma(s,X(s))\Vert_{\mathrm{HS}}^2 ds\Big|^{\frac{\nu}{4}}\right)\\
    &\leq \EE\left(\sup_{t\in [t_0,T]}\Vert X(t)-x^{\star}\Vert^{\frac{\nu}{2}}\left(\int_{t_0}^T \Vert\sigma(s,X(s))\Vert_{\mathrm{HS}}^2\right)^{\frac{\nu}{4}}\right)\\
    &\leq \frac{1}{2K}\EE\left(\sup_{t\in [t_0,T]}\Vert X(t)-x^{\star}\Vert^{\nu}\right)+\frac{K}{2}\EE\left[\left(\int_{t_0}^T \Vert\sigma(s,X(s))\Vert_{\mathrm{HS}}^2\right)^{\frac{\nu}{2}}\right]\\
    &\leq \frac{1}{2K}\EE\left(\sup_{t\in [t_0,T]}\Vert X(t)-x^{\star}\Vert^{\nu}\right)+\frac{KC^{\frac{\nu}{2}}}{2}\EE\left[\left(\int_{t_0}^T 1+\Vert X(s)-x^{\star}\Vert^2ds\right)^{\frac{\nu}{2}}\right]\\
    &\leq \frac{1}{2K}\EE\left(\sup_{t\in [t_0,T]}\Vert X(t)-x^{\star}\Vert^{\nu}\right)+\frac{KC^{\frac{\nu}{2}}}{2}T^{\frac{\nu-2}{2}}\EE\left[\left(\int_{t_0}^T (1+\Vert X(s)-x^{\star}\Vert^2)^{\frac{\nu}{2}}ds\right)\right].
\end{align*}
Substituting this into \eqref{sdn1}, we have, for a possibly different $K=K(\nu,T)$,
\begin{equation*}
    \EE\left(\sup_{t\in [t_0,T]}\Vert X(t)-x^{\star}\Vert^{\nu}\right)\leq K\left(1+\EE\left(\Vert X_0-x^{\star}\Vert^{\nu}\right)+\int_{t_0}^T \EE\left(\sup_{t\in [0,s]}\Vert X(t)-x^{\star}\Vert^{\nu}\right) ds\right).
\end{equation*}
By Gr\"onwall's inequality, we obtain  
\begin{equation*}
    \EE\left(\sup_{t\in [t_0,T]}\Vert X(t)-x^{\star}\Vert^{\nu}\right)\leq K\left(1+\EE\left(\Vert X_0-x^{\star}\Vert^{\nu}\right)\right)e^{KT}<+\infty.
\end{equation*}
Since $T>t_0$ is arbitrary, we conclude that $X\in S_{\H}^{\nu}[t_0]$.
\end{proof}
}

\subsection{On martingales}
%\begin{proposition}\label{exchange2}
% Let $T>0$ and $t\in (0,T)$. Let $g:\Omega\times\mathbb{R} \rightarrow\mathbb{R}$ and denote $g(t)\eqdef g(\cdot,t)$, if we suppose that $g$ is such that:
% \begin{equation*}
%     \mathbb{E}\left(\int_0^T |g(s)|ds\right)<\infty.
% \end{equation*}
% Then, $$\frac{d}{dt}\mathbb{E}\left(\int_0^t g(s)ds\right)=\mathbb{E}\left(g(t)\right).$$
%\end{proposition}
%\begin{proof}
%Direct from Fubini's Theorem and \cite[Theorem 4.10]{ovchinnikov_2013}.
%\end{proof}
%\begin{lemma}\label{existenceof} \cite[Lemma 2.1]{existenceof}
%Let $f:\R_+\rightarrow\R$ be continuous, and $\liminf_{t\rightarrow\infty} f(t)\neq \limsup_{t\rightarrow\infty} f(t)$. Then there exists a constant $\alpha$, satisfying $\liminf_{t\rightarrow\infty} f(t)< \alpha<\limsup_{t\rightarrow\infty} f(t)$, such that for every $\beta>0$, we can define a sequence $(t_k)_{k\in\N}\subset\R$ such that $$f(t_k)>\alpha,\hspace{0.3cm} t_{k+1}>t_k+\beta\hspace{0.3cm} \forall k\in\N.$$
%\end{lemma}

\begin{proposition}\label{burkholder1}  (see %\cite[Theorem 1.7.3]{mao},
\cite{burkholder} and \cite[Section 1.2]{fitzpatrick}) (Burkholder-Davis-Gundy Inequality)
 Let $p>0$, $W$ be a $\K$-valued cylindrical Brownian motion defined over a filtered probability space $(\Omega,\mathcal{F},\{\mathcal{F}_t\}_{t\geq 0},\mathbb{P})$ and $g:\Omega\times \R_+\rightarrow\K$ a progressively measurable process (with our usual notation $g(t)\eqdef g(\cdot,t)$) such that \[\mathbb{E}\left[\left(\int_0^T \Vert g(s)\Vert^2ds\right)^{\frac{p}{2}}\right]<+\infty, \quad \forall T>0.\]
    Then, there exists $C_p>0$ (only depending on $p$) for every $T>0$ such that:
    \[\mathbb{E}\left[\sup_{t\in [0,T]}\Bigg|\int_0^t\langle g(s),dW(s)\rangle\Bigg|^p\right]\leq C_p \mathbb{E}\left[\left(\int_0^T \Vert g(s)\Vert^2ds\right)^{\frac{p}{2}}\right].\]
\end{proposition}

%\begin{theorem}\cite{doob}\label{doob}
%Let $(M_t)_{t\geq 0}:\Omega\rightarrow\R$ be a continuous martingale such that $\sup_{t\geq 0} \EE\pa{|M_t|^p}<+\infty$ for some $p>1$. Then there exists a random variable $M_{\infty}\in\Lp^p(\Omega;\R)$ such that $\lim_{t\rightarrow +\infty} M_t= M_{\infty}$ a.s..
%\end{theorem}

\begin{theorem}\label{convmartingale}
Let $\H$ be a real separable Hilbert space and $(M_t)_{t\geq 0}:\Omega\rightarrow\H$ be a continuous martingale such that $\sup_{t\geq 0} \EE\pa{\Vert M_t\Vert^2}<+\infty$. Then there exists a $\H-$valued random variable 
 $M_{\infty}\in\Lp^2(\Omega;\H)$ such that $\slim_{t\rightarrow \infty} M_t=M_{\infty}$ a.s..   
\end{theorem}
\begin{proof}
%This is a direct consequence of \cite[Theorem 5.2.1, Theorem 4.7]{aust}, which extends to the continuous case the results from \cite[Theorem 4.3]{scalora} about the convergence of Banach-valued martingales, which in turn extends the well-known results about real-valued martingale convergence in \cite{doob}.
Consider $(M_k)_{k\in\N}$ to be the embedded discrete parameter martingale. Since \linebreak
$\sup_{k\in\N}\EE{\Vert M_k\Vert^2}<+\infty$, then $(M_k)_{k\in\N}$ is uniformly integrable and by \cite[Theorem~3]{scalora}, there exists a measurable $\H$-valued random variable $M_{\infty}\in\Lp^2(\Omega;\H)$ such that $\lim_{k\rightarrow\infty} \norm{M_k-M_{\infty}}=0$ a.s.. In turn, using the dominated convergence theorem (see \cite[Theorem 1.34]{rudin}), we also have \begin{equation}\label{convl2}
    \lim_{k\rightarrow\infty}\EE(\Vert M_k-M_{\infty}\Vert^2)=0.
\end{equation}

%Furthermore, by the almost sure convergence and Fatou's Lemma: $$\EE(\Vert M_{\infty}\Vert^2)=\EE(\liminf_{k\rightarrow\infty}\Vert M_k\Vert^2)\leq \liminf_{k\rightarrow\infty}\EE(\Vert M_k\Vert^2)<+\infty,$$
%or equivalently $M_{\infty}\in\Lp^2(\Omega;\H)$.\smallskip
The rest of the proof is inspired by the arguments in the proof of \cite[Theorem~2.2]{ghafari}.\smallskip

We consider an arbitrary $k\in\N^*$ and $\delta>0$. Since $(M_{t+k}-M_k)_{t\geq 0}$ is also a $\H-$valued martingale, we can use Doob's maximal inequalities for $\H-$valued martingales shown in \cite[Theorem~2.2]{infinite}, which gives us 
\begin{equation}\label{marrr}
\delta^2\Pro\left(\sup_{s\in [0,t]}\Vert M_{s+k}-M_k\Vert>\delta\right)\leq \EE(\Vert M_{t+k}-M_k\Vert^2).
\end{equation}

Let $n\in \N^*$ be arbitrary. We have
\begin{align*}
\Pro\left(\sup_{s\in \Q\cap[0,n]}\Vert M_{s+k}-M_{\infty}\Vert>\delta\right)\leq \Pro\left(\sup_{s\in \Q\cap[0,n]}\Vert M_{s+k}-M_k\Vert>\frac{\delta}{2}\right)+ \Pro\left(\Vert M_{k}-M_{\infty}\Vert>\frac{\delta}{2}\right).
\end{align*}
Using \eqref{marrr} and Markov's inequality, we get the bound 
\begin{equation}
    \begin{aligned}\label{marrr2}
    \delta^2\Pro\left(\sup_{s\in \Q\cap[0,n]}\Vert M_{s+k}-M_{\infty}\Vert>\delta\right)&\leq 4\EE(\Vert M_{n+k}-M_k\Vert^2)+4\EE(\Vert M_k-M_{\infty}\Vert^2)\\
    &\leq 8\EE(\Vert M_{n+k}-M_{\infty}\Vert^2)+12\EE(\Vert M_k-M_{\infty}\Vert^2) .
\end{aligned}
\end{equation}

In turn, we get
\begin{align*}
\delta^2\Pro\left(\sup_{s\in\Q, s\geq k}\Vert M_s-M_{\infty}\Vert>\delta\right)
&\leq \delta^2 \Pro\left(\bigcup_{n\in \N^*}\Big\{\sup_{s\in\Q\cap [0,n]}\Vert M_{s+k}-M_{\infty}\Vert>\delta\Big\}\right)\\
&\leq \delta^2\liminf_{n\rightarrow\infty}\Pro\left(\sup_{s\in\Q\cap [0,n]}\Vert M_{s+k}-M_{\infty}\Vert>\delta\right)\\
&\leq 12 \EE\pa{\Vert M_k-M_{\infty}\Vert^2}, 
\end{align*}
where we have used \eqref{marrr2} and in the last inequality, that $\lim_{n\rightarrow\infty} \EE(\Vert M_{n+k}-M_{\infty}\Vert^2)=0$ by \eqref{convl2}. Taking $k\rightarrow\infty$, and using again \eqref{convl2}, we conclude that for all $\delta > 0$ 
\begin{align*}
\lim_{k\rightarrow\infty}\Pro\left(\sup_{s\in\Q, s\geq k}\Vert M_s-M_{\infty}\Vert>\delta\right)=0.
\end{align*}
For $k\in\N^*$, we define $A_k \eqdef \{\omega\in\Omega:\sup_{s\in\Q, s\geq k}\Vert M_s(\omega)-M_{\infty}(\omega)\Vert>\delta\}$, since $(A_k)_{k\in\N^*}$ is a non-increasing sequence of sets: 
\[
0=\lim_{k\rightarrow\infty}\Pro\left(A_k\right)=\Pro\left(\bigcap_{k\in\N^*}A_k\right).
\]
Defining for $l\geq 0$, $D_l=\{\omega\in\Omega:\Vert M_l(\omega)-M_{\infty}(\omega)\Vert>\delta\}$, it is direct to check that $\bigcup_{l\geq k, l\in \Q}D_l\subseteq A_k$ for every $k\in\N^*$. Therefore, we obtain that
\[
\Pro\left(\bigcap_{k\in\N^*}\bigcup_{l\geq k, l\in \Q}D_l\right)=0,
\] 

which is equivalent to $\slim_{s \rightarrow \infty, s \in \Q} M_t=M_{\infty}$ a.s.. The result follows from classical arguments of continuity of the martingale.
\end{proof}

\begin{theorem} \label{impp} \cite[Theorem 1.3.9]{mao}
 Let $\{A_t\}_{t\geq 0} $ and $\{U_t\}_{t\geq 0} $ be two continuous adapted increasing processes with $A_0=U_0=0$ a.s.. Let $\{M_t\}_{t\geq 0} $ be a real-valued continuous local martingale with $M_0=0$ a.s.. Let $\xi$ be a nonnegative $\mathcal{F}_0$-measurable random variable. Define  \[X_t=\xi+A_t-U_t+M_t\hspace{0.3cm} \text{ for } t\geq 0.\]
 If $X_t$ is nonnegative and $\lim_{t\rightarrow\infty} A_t<\infty$, then   $\lim_{t\rightarrow\infty} X_t$ exists and is finite, and $\lim_{t\rightarrow\infty} U_t<\infty$.
\end{theorem}

\bibliographystyle{unsrt}
\smaller
\bibliography{samplebib}

\end{document}